\documentclass[a4paper, 12pt]{article}
\usepackage{amsmath}
\usepackage{amsthm}
\usepackage{amssymb}
\usepackage{latexsym}
\usepackage{enumerate}
\usepackage{graphics}
\usepackage{array}
\usepackage[dvips]{graphicx}
\usepackage{epsfig}
\usepackage{enumerate}
\usepackage[english]{babel}
\usepackage{amsmath,amssymb,amsthm}
\usepackage{lineno}

\newcommand*{\Scale}[2][4]{\scalebox{#1}{$#2$}}%
\usepackage{tikz}
\usetikzlibrary{shapes,arrows}
\tikzstyle{place}=[circle,draw=blue!50,fill=blue!20,thick]
\tikzstyle{place1}=[circle,draw=blue!2,fill=blue!20,thick]
\tikzstyle{placetab}=[circle,draw=blue!50,fill=blue!20,thick, inner sep=0pt,minimum size=2mm]

\newcommand{\be}{\begin{equation}}
\newcommand{\ee}{\end{equation}}
\newcommand{\bea}{\begin{eqnarray}}
\newcommand{\eea}{\end{eqnarray}}

\newcommand{\nod}{\noindent}

\newcommand{\ba}{\begin{array}}
\newcommand{\ea}{\end{array}}
\newcommand{\bc}{\begin{center}}
\newcommand{\ec}{\end{center}}

\newtheorem{theo}{Theorem}

\newtheorem{defi}{Definition}

\newtheorem{rem}{Remark}

\newcommand{\beq}{\begin{equation}}
\newcommand{\eeq}{\end{equation}}
\newcommand{\la}{\langle}
\newcommand{\ra}{\rangle}

\allowdisplaybreaks

\usepackage[hmargin=2.0cm,vmargin=2.0cm]{geometry}
\begin{document}
%\linenumbers*[1]
%\title{\bf Exact and approximate epidemic models on networks: a new, improved closure relation\\}
\title{\bf Exact deterministic representation of Markovian $SIR$ epidemics on networks with and without loops\\}

\author{Istvan Z. Kiss$^{1,\ast}$, Charles G. Morris$^{1}$, Fanni S\'elley$^{2}$,\\
P\'eter L. Simon$^{2}$ \& Robert R. Wilkinson$^{3}$}

\maketitle

\begin{center}

$^1$ School of Mathematical and Physical Sciences, Department of
Mathematics, University of Sussex, Falmer,
Brighton BN1 9QH, UK\\
$^2$ Institute of Mathematics, E\"otv\"os Lor\'{a}nd University
Budapest, and Numerical Analysis and Large Networks Research Group, Hungarian Academy of Sciences, Hungary\\
$^3$ Department of Mathematical Sciences, The University of Liverpool, Peach Street, Liverpool L69 7ZL, UK

\end{center}

\vspace{8cm}
\begin{flushleft}
$\ast$ corresponding author\\
email: i.z.kiss@sussex.ac.uk\\
\end{flushleft}

\newpage
\begin{abstract}
In a previous paper Sharkey et al. \cite{Sharkeyetal13Exact} proved the exactness of closures at the level of triples for Markovian $SIR$ (susceptible-infected-removed) dynamics on tree-like networks. This resulted in a deterministic representation of  the epidemic dynamics on the network that can be numerically evaluated. In this paper, we extend this modelling framework to certain classes of networks exhibiting loops. We show  that closures where the loops are kept intact are exact, and lead to a simplified and numerically solvable system of ODEs (ordinary-differential-equations). The findings of the paper lead us to a generalisation of closures that are based on partitioning the network around nodes that are cut-vertices (i.e. the removal of such a node leads to the network breaking down into at least two disjointed components or subnetworks). Exploiting this structural property of the network yields some natural closures, where the evolution of a particular state can typically be exactly given in terms of the corresponding or projected sates on the subnetworks
and the cut-vertex. A byproduct of this analysis is an alternative probabilistic proof of the exactness of the closures for tree-like networks presented in Sharkey et al. \cite{Sharkeyetal13Exact}.  In this paper we also elaborate on how the main result can be applied to more realistic networks, for which we write down the ODEs explicitly and compare output from these to results from simulation. Furthermore, we give a general, recipe-like method of how to apply the reduction by closures technique for arbitrary networks, and give an upper bound on the maximum number of equations needed for an exact representation.

\end{abstract}

\nod {\bf Keywords:} master equation, network, closure, loop, cut-vertex

\newpage

%%%%%%%%%%%%%%%%%%%%%%%%%%%%%%%%%%%%%%%%%%%%%%%%%%%%%%%%%%%%%%%%%%%%%%%%%%%%%%%%%%%%%%%%%%
\section{Introduction}
%%%%%%%%%%%%%%%%%%%%%%%%%%%%%%%%%%%%%%%%%%%%%%%%%%%%%%%%%%%%%%%%%%%%%%%%%%%%%%%%%%%%%%%%%%
Despite tremendous progress over the past decade or so, modelling transmission processes on networks still poses many challenges.
A significant number of such models are concerned with modelling epidemics on networks, in particular $SIR$ dynamics which makes the 
treatment of some models easier due to the linear transmission process, as opposed to $SIS$ dynamics where nodes can become reinfected multiple times.
There is a wealth of modelling approaches to this problem \cite{LeonReview, HouseUnifyingApp, KarrerNewman2010Exact, TaylorKissJMB} that differ in the choice of variables at which models are formulated, and whether averages are taken at the population level, or a probabilistic view is kept whereby either the full state space is considered \cite{SimonKissAut}, or where, modelling starts at node level \cite{Sharkey08, Sharkey11, Sharkeyetal13Exact}.

A major further challenge is posed by extending existing results for loopless networks to networks with loops or clustered networks, where clustering and the presence of loops is closely related (i.e. the presence of many closed loops of size three leads to high levels of clustering). The specific issues cluster around the generation of synthetic networks with tuneable clustering \cite{GreenKissLargeScaleClust, Newman2003ClustNetwGeneration, VolzTuneDegreeClust}, as well as the development of low-dimensional approximate or exact models with the aim to match output directly from the stochastic process. Some progress in both areas has been made, with final epidemic size calculations (non-time-dependent measures) giving excellent agreement for specific clustered networks \cite{Ball, GleesonClusteredNetworks, Newman2009RanNetwClust}. There are also examples for good time-evolution models \cite{HouseUnifyingApp, VolzPlosOne}, but many important and difficult questions remain. 

Using the approach introduced in \cite{Sharkey08,Sharkey11,Sharkeyetal13Exact}, we present exact, deterministic representations of Markovian $SIR$ epidemics on networks with and without loops and identify the link between the structural properties of the networks and the viability of closures that allow us to
write down exact systems that can be numerically evaluated. Here, the equations start at the level of nodes and consider the exact probability of nodes being
susceptible, infected or recovered at a given time, see \cite{Sharkey08, Sharkey11, Sharkeyetal13Exact}. In particular, we show the link between nodes that are cut-vertices and edges that are bridges (both defined later), and the feasibility of closures. Assuming a network with $N$ nodes with the weighted connectivity and transmissibility rate matrix given by 
$T=(T_{ij})_{i,j=1, 2, \dots, N}$, where $T_{ii}=0 \, \forall i=1,2, \dots, N$, the evolution equations for singles and pairs are given by,
\begin{eqnarray}\nonumber
\dot{\la S_i\ra}&=&-\sum_{j=1}^N T_{ij}\la S_iI_j\ra, \\ \nonumber
\dot{\la I_i\ra}&=&\sum_{j=1}^N
T_{ij}\la S_iI_j\ra -\gamma_i\la I_i\ra,\\ \nonumber
\dot{\la S_iI_j\ra}&=&\sum_{k=1,k\neq i}^N T_{jk}\la S_iS_jI_k\ra-\sum_{k=1, k\neq j}^N T_{ik}\la I_kS_iI_j\ra \\  \nonumber
&& -T_{ij}\la S_iI_j\ra-\gamma_j\la S_iI_j\ra, \\
\dot{\la S_iS_j\ra}&=&-\sum_{k=1, k\neq j}^NT_{ik}\la I_kS_iS_j\ra-\sum_{k=1, k\neq i}^NT_{jk}\la
S_iS_jI_k\ra, \label{SharkeyOriginal}
\end{eqnarray}
where $\la A_i\ra$ denotes the time-dependent probability for individual $i$ being
in state $A$, and expressions of the form $\la A_iB_j\ra$ denote the
time-dependent probability that individuals $i$ and $j$ are in states $A$ and $B$, respectively.
We assume that all processes, i.e. infection and recovery, are independent Poisson processes, with per-link infection rate denoted by $\tau$ and absorbed in $T$ and rates of 
recovery $\gamma_j$ ($j=1,2, \dots, N$). While this is a general model formulation from the network view point, in this paper all numerical simulations are carried out using undirected and unweighted networks with the per-contact transmission rate specified explicitly, and with the same recovery rate for all nodes.

The system above is not closed as equations for the triples are needed. In Sharkey et al. \cite{Sharkeyetal13Exact}, the authors have proved that for tree-like networks and for some 
special cases of non-tree-like networks the following closures hold and are exact
\beq
\la S_j\ra\la S_iS_jI_k\ra = \la S_iS_j\ra\la S_jI_k\ra, \label{closure1}
\eeq
for all $i, j, k \in\{1,2,\dots,N\}$, and for all $j$ with links towards $i$ and all $k$ with links towards $j$ and $i\neq k$ (i.e. $T_{ij} \neq 0$ and $T_{jk} \neq 0$), and
\beq
\la S_i\ra\la I_kS_iI_j\ra = \la I_kS_i\ra\la S_i I_j\ra,\label{closure2}
\eeq
for all $i, j, k \in\{1,2, \dots ,N\}$ and for all $k$ and $j$ with links towards $i$ (i.e. $T_{ik} \neq 0$ and $T_{ij} \neq 0$), and $j \neq k$.
Closures are exact in the sense that the closed system given below,
\begin{eqnarray}\nonumber
\dot{\la X_i\ra}&=&-\sum_{j=1}^N T_{ij}\la X_iY_j\ra, \\ \nonumber 
\dot{\la Y_i\ra}&=&\sum_{j=1}^N
T_{ij}\la X_iY_j\ra -\gamma_i\la Y_i\ra,\\ \nonumber
\dot{\la X_iY_j\ra}&=&\sum_{k=1,k\neq i}^N T_{jk}\frac{\la X_iX_j\ra\la X_jY_k\ra}{\la X_j\ra}-\sum_{k=1,k\neq j}^N T_{ik}\frac{\la X_iY_k\ra\la X_iY_j\ra}{\la X_i\ra}    \\ \nonumber
&&-T_{ij}\la X_iY_j\ra-\gamma_j\la X_iY_j\ra, \\
\dot{\la X_iX_j\ra}&=&-\sum_{k=1,k\neq j}^NT_{ik}\frac{\la Y_kX_i\ra\la X_iX_j\ra}{\la
X_i\ra}-\sum_{k=1,k\neq i}^NT_{jk}\frac{\la X_iX_j\ra\la X_jY_k\ra}{\la X_j\ra}, \label{0.2}
\label{0.31}
\end{eqnarray}
is such that when $\la X_i\ra=\la S_i\ra$, $\la Y_i\ra=\la I_i\ra$ hold at $t=0$, then these will hold for $\forall t >0$. Similar
equalities hold for the pairs.  To emphasise that this second system is an `approximation', we use $X$ for susceptible and $Y$ for infected.
The main result of the original paper \cite{Sharkeyetal13Exact}, on which we now build, is the proof that for a tree-like network, the closure is exact in the sense that solving the closed 
system we get the same values for the probabilities.

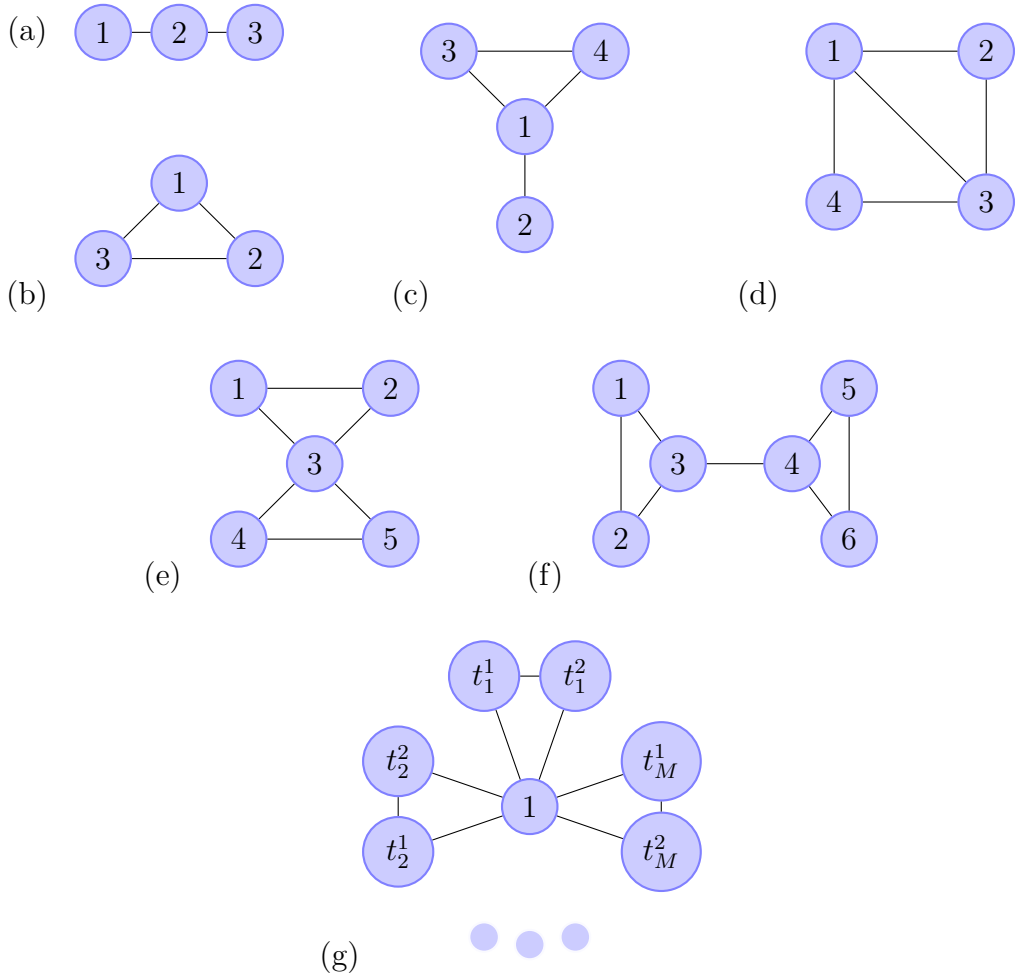
\begin{figure}
 	\begin{center}
		% open and closed triangle
		\begin{tikzpicture}
			\node[place] (left) at ( -1,1) {$1$};
			\node[place] (middle) at (0,1) {$2$};
			\node[place] (right) at (1,1) {$3$};
			\node (a) at (-2,1) {(a)};
			\node[place] (leftbase) at (-1,-2) {$3$};
			\node[place] (rightbase) at (1,-2) {$2$};
			\node[place] (middletop) at (0,-1) {$1$};
			\node (a) at (-2,-2.5) {(b)};			
			\draw [-] (left) -- (middle);
			\draw [-] (right) -- (middle);
			\draw [-] (leftbase) -- (rightbase);
			\draw [-] (rightbase) -- (middletop);			
			\draw [-] (leftbase) -- (middletop);					
		\end{tikzpicture}
		\hspace{1.0cm}
		%\vspace{-1.5cm}
		% lollipop
		\begin{tikzpicture}
			\node[place] (lefttop) at ( -1,1) {$3$};
			\node[place] (middle) at ( 0,0) {$1$};
			\node[place] (righttop) at ( 1,1) {$4$};
			\node[place] (bottom) at (0,-1.3) {$2$};
			\node (a) at (-1.5,-2.25) {(c)};
			\draw [-] (lefttop) -- (middle);
			\draw [-] (righttop) -- (middle);
			\draw [-] (lefttop) -- (righttop);
			\draw [-] (middle) -- (bottom);			
		\end{tikzpicture}
		\hspace{1.0cm}
			% toast
			\begin{tikzpicture}
			\node[place] (lefttop) at ( -1,1) {$1$};
			\node[place] (righttop) at ( 1,1) {$2$};
			\node[place] (leftbottom) at ( -1,-1) {$4$};
			\node[place] (rightbottom) at (1,-1) {$3$};
			\node (a) at (-2.0,-2.25) {(d)};
			\draw [-] (lefttop) -- (righttop);
			\draw [-] (righttop) -- (rightbottom);
			\draw [-] (lefttop) -- (leftbottom);
			\draw [-] (leftbottom) -- (rightbottom);
			\draw [-] (lefttop) -- (rightbottom);			
		\end{tikzpicture}
	\end{center}
	\begin{center}	
		% bow tie
		\begin{tikzpicture}
			\node[place] (lefttop) at ( -1,1) {$1$};
			\node[place] (middle) at ( 0,0) {$3$};
			\node[place] (righttop) at ( 1,1) {$2$};
			\node[place] (leftbottom) at (-1,-1) {$4$};
			\node[place] (rightbottom) at (1,-1) {$5$};
			\node (a) at (-2.0,-1.5) {(e)};
			\draw [-] (lefttop) -- (middle);
			\draw [-] (righttop) -- (middle);
			\draw [-] (lefttop) -- (righttop);
			\draw [-] (middle) -- (leftbottom);	
			\draw [-] (middle) -- (rightbottom);
			\draw [-] (rightbottom) -- (leftbottom);					
		\end{tikzpicture}
		\hspace{1.0cm}
			% bow tie with a bridge
			\begin{tikzpicture}
			\node[place] (lefttop) at ( -1.5,1) {$1$};
			\node[place] (leftbottom) at (-1.5,-1) {$2$};
			\node[place] (leftmiddle) at ( -0.75,0) {$3$};
			\node[place] (righttop) at ( 1.5,1) {$5$};
			\node[place] (rightbottom) at (1.5,-1) {$6$};
			\node[place] (rightmiddle) at (0.75,0) {$4$};
			\node (a) at (-2.5,-1.5) {(f)};
			\draw [-] (lefttop) -- (leftbottom);
			\draw [-] (lefttop) -- (leftmiddle);
			\draw [-] (leftbottom) -- (leftmiddle);
			\draw [-] (rightmiddle) -- (leftmiddle);
			\draw [-] (righttop) -- (rightbottom);
			\draw [-] (righttop) -- (rightmiddle);
			\draw [-] (rightbottom) -- (rightmiddle);
			%\draw [-] (rightmoddle) -- (rightbottom);
			\end{tikzpicture}
		\end{center}			
		\begin{center}
			% star triangle
			\begin{tikzpicture}
			\node[place] (lefttop) at ( -0.6,1.73) {$t_1^1$};
			\node[place] (middle) at ( 0,0) {$1$};
			\node[place] (righttop) at ( 0.6,1.73) {$t_1^2$};
			\node[place] (leftlefttop) at (-1.73,0.6) {$t_2^2$};
			\node[place] (leftleftbottom) at (-1.73,-0.6) {$t_2^1$};
			\node[place] (rightrighttop) at (1.73,0.6) {$t_M^1$};
			\node[place] (rightrightbottom) at (1.73,-0.6) {$t_M^2$};
			\node[place1] (dot1) at (-0.6,-1.73) {$$};
			\node[place1] (dot2) at (0,-1.83) {$$};
			\node[place1] (dot3) at (0.6,-1.73) {$$};
			\node (a) at (-2.5,-2.0) {(g)};
			\draw [-] (lefttop) -- (middle);
			\draw [-] (righttop) -- (middle);
			\draw [-] (lefttop) -- (righttop);
			\draw [-] (leftlefttop) -- (middle);
			\draw [-] (leftleftbottom) -- (middle);
			\draw [-] (leftlefttop) -- (leftleftbottom);
			\draw [-] (rightrighttop) -- (middle);
			\draw [-] (rightrightbottom) -- (middle);
			\draw [-] (rightrighttop) -- (rightrightbottom);
		\end{tikzpicture}
	\end{center}
	\caption{Open triangle or line network of three nodes (a), closed triangle loop (b), lollipop (c), toast (d), bow tie (e), bow tie with a bridge (f), and star-triangle (i.e. a star of 
	triangles) (g) network with $M$ triangles.}\label{lollipop-bowtie}
\end{figure}			

Regarding the original exact equations (Eq.~(\ref{SharkeyOriginal})) we make the following remarks:
\begin{rem}
\noindent \begin{enumerate}
\item The equations emerge naturally starting from nodes and building up to higher moments.
The equations do not cover every possible configuration of states across connected subnetworks. For example, the exact equations
do not require knowledge of pairs such as $\la I_{i}I_{j}\ra$ as these are not required by the system dynamics.
\item All the triples that appear are such that the middle node is susceptible. In a tree-like network this effectively means that the nodes to the `left' and to the `right' of the middle node 
have not yet communicated since the only way is through the middle susceptible node, and thus, technically their states are independent.
\item The closures will require extra variables that have previously not been needed, for example, evolution equations for $(SS)$ pairs (e.g. $\la S_{i}S_{j}\ra$) are needed. This
is required by the closure of an $(SSI)$ triple (e.g. $\la S_iS_jI_k\ra$).
\item Additional closures above and beyond what the equations require and emerge naturally may not hold, but this is not relevant as these are not needed to
derive a closed, exact system with fewer equations.
\end{enumerate}
\end{rem}

%%%%%%%%%%%%%%%%%%%%%%
\section{Background and examples}
%%%%%%%%%%%%%%%%%%%%%
An $SIR$ epidemic on an arbitrary directed and weighted network can be represented as a set of equations starting at the level of individuals and by building up to and accounting for
the dependencies of these on pairs, and of the pairs on triples, and so on, until full system size is reached. At this point the system of equations will become
well-defined and self-consistent. The aim in this type of modelling approach is to find closures, where higher order moments can be approximated or specified exactly in terms of
combinations of lower order moments. It is now well-known and accepted that this is feasible for tree-like networks and for networks with loops, but starting from some very specific 
initial conditions or particular example networks \cite{KarrerNewman2010Exact, Sharkeyetal13Exact}. Here, we will show a more general approach to extend the ideas of closures to networks with loops/cycles and we will also consider specific fully worked-out examples, as well as how this approach could be generalised to and its feasibility for a larger class of networks.\\

\begin{figure}
\begin{center}
\epsfig{file=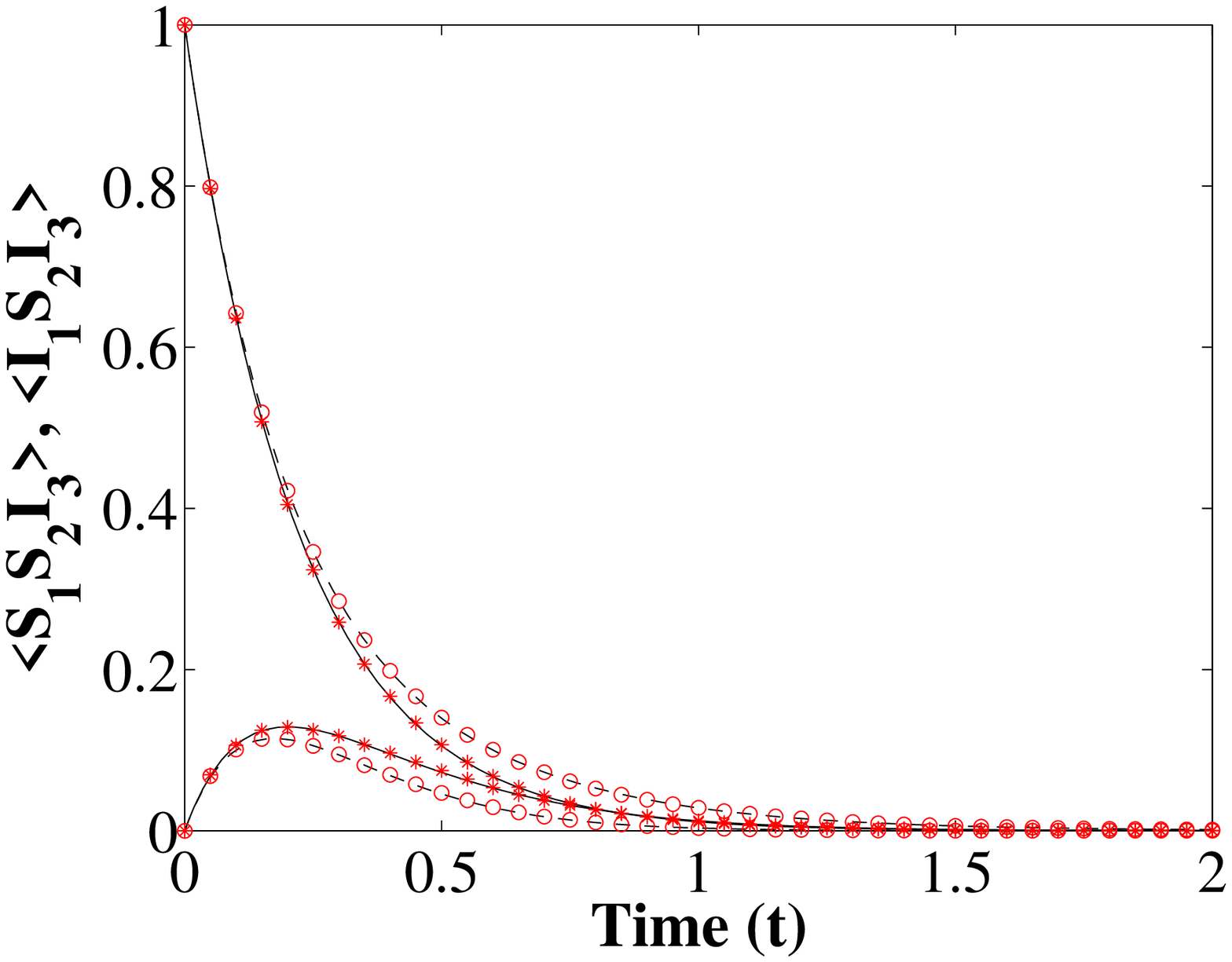,width=10cm}
\end{center}
\caption{Testing of two closures on the fully connected triangle with 3 nodes (see Fig.~\ref{lollipop-bowtie}b), ($\la S_1S_2I_3\ra=\frac{\la S_1S_2 \ra \la S_2I_3\ra}{\la S_2 \ra}$ and $\la I_1S_2I_3\ra=\frac{\la I_1S_2 \ra \la S_2I_3\ra}{\la S_2 \ra}$), by plotting  the left and right hand sides of the closures from the full system. Continuous lines represent the left hand sides, and dashed lines represent the right hand sides for the ($SSI$) and ($ISI$) triples. Results from the average of $10^5$ Gillespie-type simulations are plotted by ($\star$) and ($\circ$). System started such that  $\la S_1S_2I_3\ra(0)=1$, and hence, $\la I_1S_2I_3\ra(0)=0$. Parameter values for all cases are $\tau=7/4$ and $\gamma=1$.}
\label{BadCloTri}
\end{figure}

%%%%%%%%%%%%%%%%%%%%%%%%%%%%%%%%%%%%%
\noindent \textbf{Testing closures: examples where closures do and do not work\\}
%%%%%%%%%%%%%%%%%%%%%%%%%%%%%%%%%%%%%
Before we generalise the ideas of closures to networks with loops, let us consider the special cases of small networks with and without loops, namely the open and closed triple, lollipop, and toast networks presented in Fig.~\ref{lollipop-bowtie}. For such small networks, the full system of equations can be written down and closures can be tested. 
The case of an open triple, or a line network of three nodes, is discussed explicitly in \cite{Sharkeyetal13Exact}, where an analytical proof for the exactness of the closure is presented with calculations relying explicitly on the full-system of equations. The simplest and most obvious network with loops (or a loop) is the closed triangle, see Fig.~\ref{lollipop-bowtie}b. In Appendix \ref{TriangleEqs}, we list the full system of differential equations, and we use this to show that triple closures, such as those presented in Eqs.~(\ref{closure1}-
\ref{closure2}), do not hold for closed loops. In Fig.~\ref{BadCloTri} we show that the ($SSI$)- and ($ISI$)-type triples cannot be closed exactly. This is based on numerically evaluating 
the full system, and thus, being able to compare the closure via the exact time evolution of the closures' constituent parts. This simple analysis suggests that loops cannot be closed, and it is likely that, for exact closures loops need to be kept intact. 

Keeping the loops intact is a feasible approach and we illustrate this for the lollipop network (Fig.~\ref{lollipop-bowtie}c). We start by generating and writing down the full set of 
equations for the lollipop network, see Appendix \ref{LolEqs}. These include two sets of equations. First, the naturally emerging set of equations which can be broken down into those that can, Eqs.~(\ref{cannotbeclosed_first}-\ref{cannotbeclosed_last}), and cannot be closed, Eqs.~(\ref{canbeclosed_first}-\ref{canbeclosed_last}). These together give rise to the natural full system (NFS).  Second, the non-closable equations together with the extra variables required by the closures give rise to a reduced system (RS), see Eqs.~(\ref{cannotbeclosed_first}-\ref{cannotbeclosed_last}) plus Eqs.~(\ref{extra_by_closure_first}-\ref{extra_by_closure_last}). Figure~\ref{test_lollipop_closure} gives clear numerical evidence that closures at the full system level that keep the closed triangle complete are exact (see the left panel of Fig.~\ref{test_lollipop_closure}). Moreover, in the same figure (see right panel), we compare the prevalence resulting from the NFS and the RS, and this clearly illustrates that the proposed closures are likely to be exact. By looking at the full set of equations for the lollipop network one can notice that node 1 always appears as a susceptible node whenever appearing in a triple or quadruple which is a candidate for closure. This is an important structural property of the node within the network that will be elaborated later on. We also note that two main types of closures emerge: (a) closures at the level of triples which are not part of the closed triangle and (b) closures at the level of the full system in a way in which the quad is broken down into the fully connected triangle and the tail-edge of the lollipop.

Moving towards networks with more or multiple loops we consider the toast network (see Fig.~\ref{lollipop-bowtie}d), with full equations for this given in Appendix \ref{ToastEqs}. In Fig.~\ref{test_toast_closure}, we give examples based on the full system and simulation. The left panel of the figure clearly shows that the proposed closure does not hold. For the toast network the problem of the closure is more complex as it is possible to write down 
closures that hold, see the right panel of Fig.~\ref{test_toast_closure}. However, the reduction in the number of equations due to these closures is not significant. Hence, again it is clear that the equations cannot be closed, the loops and the whole network need to be kept intact.

Several tests can be performed to test the validity of the closures. First, the validity of the closures can be tested directly from the full system using the NFS. Second, the NFS and RS 
can be compared via the evolution of prevalence in time. We also note, that when the full system is available it is possible to give an analytic proof that closures hold. This involves 
rearranging the closure relation as a difference, for example as $\alpha(t)=\la \cdot \ra \la \cdot \ra - \la \cdot \ra \la \cdot \ra$. This is then followed by showing that $\dot{\alpha}(t)=0 \forall t \geq 0$, which coupled with the closure holding at time $t=0$ gives the desired result. The calculations will involve other closure-like or $\alpha$-like expressions, but it will be possible to show that all such closures are such that they satisfy a $\dot{\alpha}(t)=-C \alpha(t)$ equation, where $C>0$, see Sharkey et al. \cite{Sharkeyetal13Exact}. For larger systems (for example the bow tie in Fig.~\ref{lollipop-bowtie}e), the terms entering the  closures can be evaluated from direct stochastic simulations and compared as such, see simulation examples in Figs.~\ref{BadCloTri} and Fig.~\ref{test_toast_closure}. 

\begin{figure}
\epsfig{file=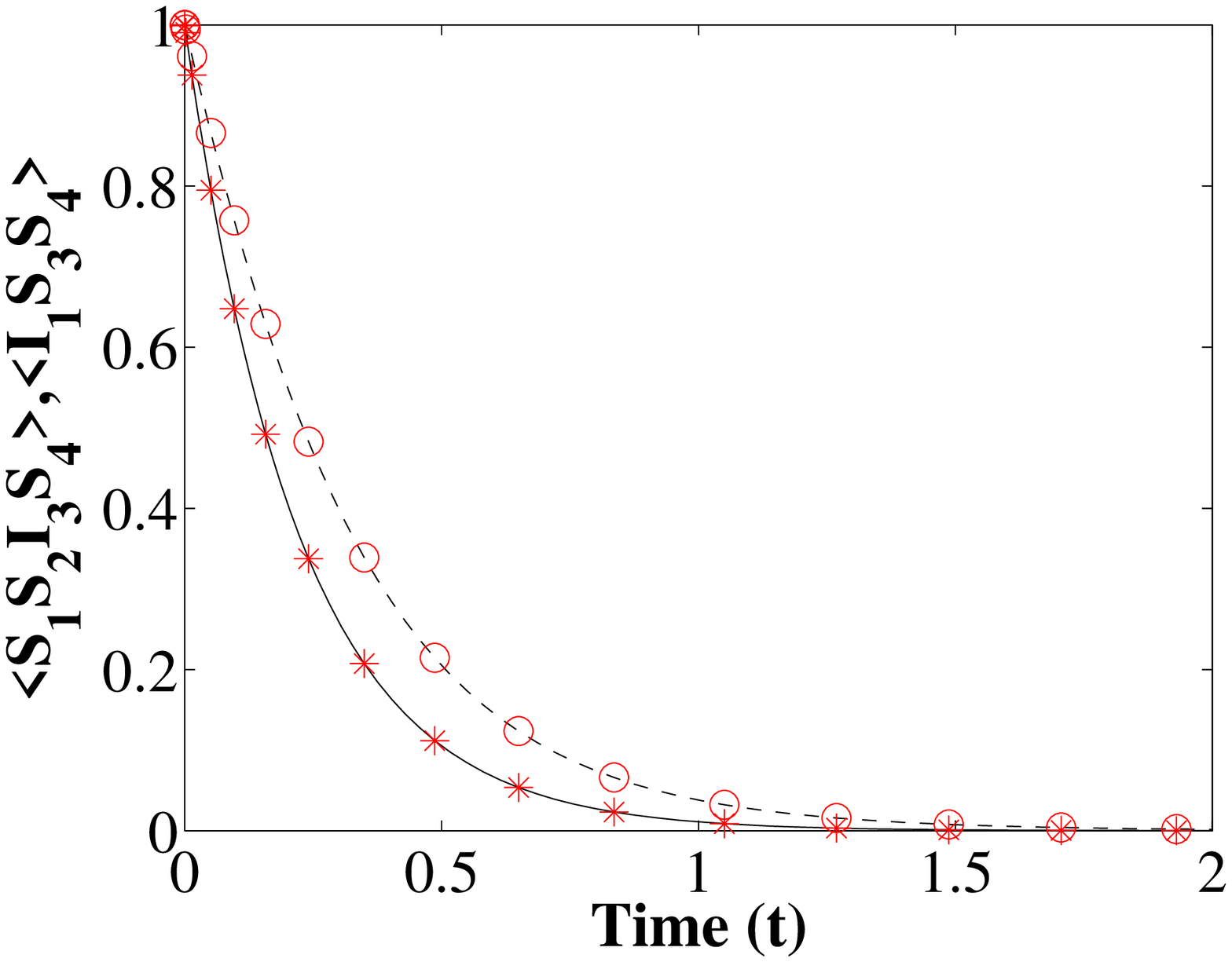,width=8cm}
\hspace{1cm}
\epsfig{file=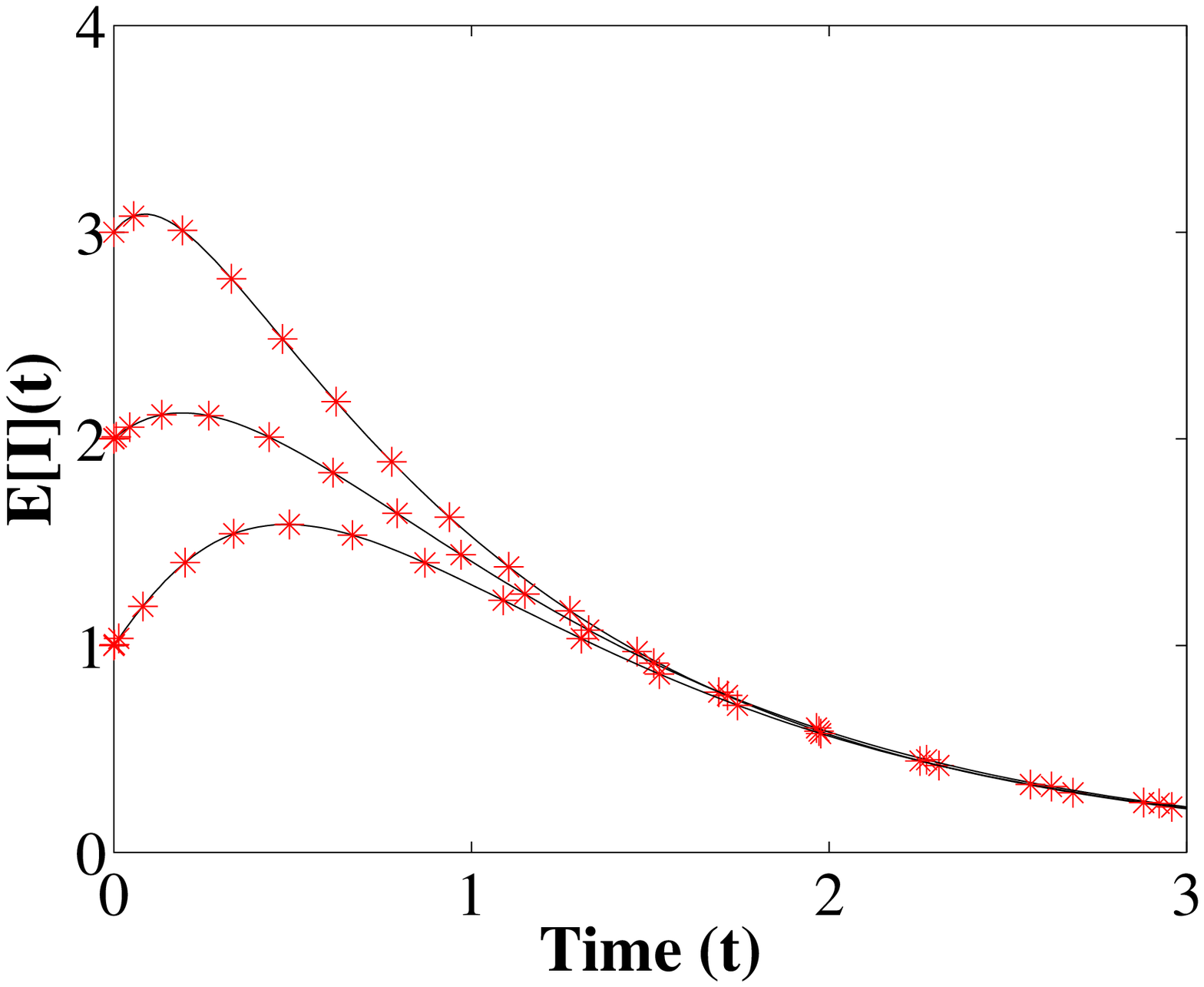,width=8cm}
\caption{\textit{Left panel:} Testing of two closures for the lollipop network (see Fig.~\ref{lollipop-bowtie}c), $\la S_1S_2I_3S_4\ra=\frac{\la S_1S_2 \ra \la S_1I_3S_4\ra}{\la S_1 \ra}
$ and $\la S_1S_2I_3\ra=\frac{\la S_1S_2 \ra \la S_1I_3\ra}{\la S_1 \ra}$, by plotting the left and right hand sides of the closures from the full system. Continuous and dashed lines represent the left hand sides of the quadruple and triangle, and ($\star$) and ($\circ$) represent the right hand sides for the corresponding closures. System started such that  $\la S_1S_2I_3S_4\ra(0)=1$. \textit{Right panel:} The expected prevalence over time for the lollipop network. Starting from three different initial conditions: (a) $\la S_1I_2I_3I_4\ra(0)=1$, (b) $\la S_1S_2I_3I_4 \ra(0)=1$ and (c) $\la 
S_1S_2I_3S_4 \ra(0)=1$. Parameter values for all cases are $\tau=7/4$ and $\gamma=1$.}\label{test_lollipop_closure}
\end{figure}

Obviously, the main role of closures is to reduce the number of equations and this can be successfully achieved for small or other networks with simple structure, such as the line and star networks. The reduction in equation numbers is illustrated in Table \ref{NoEqRedClo}. This is by no means and exhaustive list, but simply highlights the potential benefits of good closures. While the more theoretical approach of using full system equations provides a platform to test the validity of our intuition, it is not practical for networks of realistic size. Clearly, this avenue is only useful for simple, toy examples. However, numerical simulation provides an alternative, and for the case of larger networks we can test  potential closures without the need of writing down a large set of self-consistent equations. For example, based on the intuition gained so far, we test the validity of some plausible closures on the bow tie network (see Fig.~\ref{lollipop-bowtie}e). Namely we consider the following closures: 
\be
\la I_1S_2S_3I_4\ra=\frac{\la I_1S_2S_3\ra \la S_3I_4\ra}{\la S_3\ra} \,\, \text{and} \,\, \la S_2S_3I_5\ra=\frac{\la S_2S_3\ra \la S_3I_5\ra}{\la S_3\ra} .
\eeq
Testing their validity is simply a matter of numerically evaluating the probability of parts of the network being in particular states at given times. This amounts to recording the presence or otherwise of given state configurations across a fixed set of nodes. Averaging over sufficient simulations, provides an excellent approximation of the desired probability of observing a given state configuration across a given part of the network. Figure~\ref{GoodCloBT} shows clearly that the candidate closures are likely to be exact.

The simple analysis thus far, suggests that loops cannot be closed by breaking them down to their component parts. 
However, the alternative, where the loops are kept closed can be considered.\\

\begin{figure}
\epsfig{file=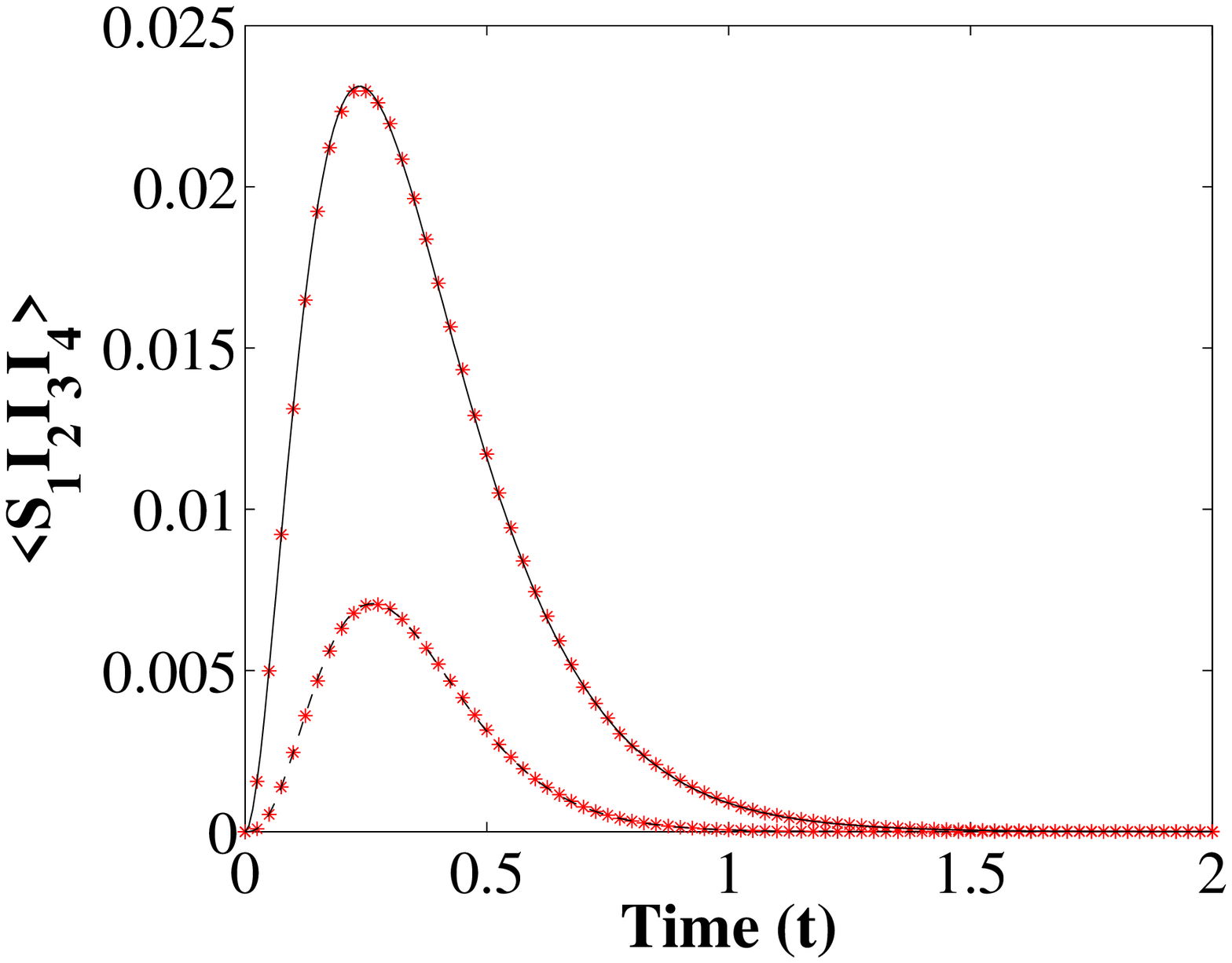,width=8cm}
\hspace{1cm}
\epsfig{file=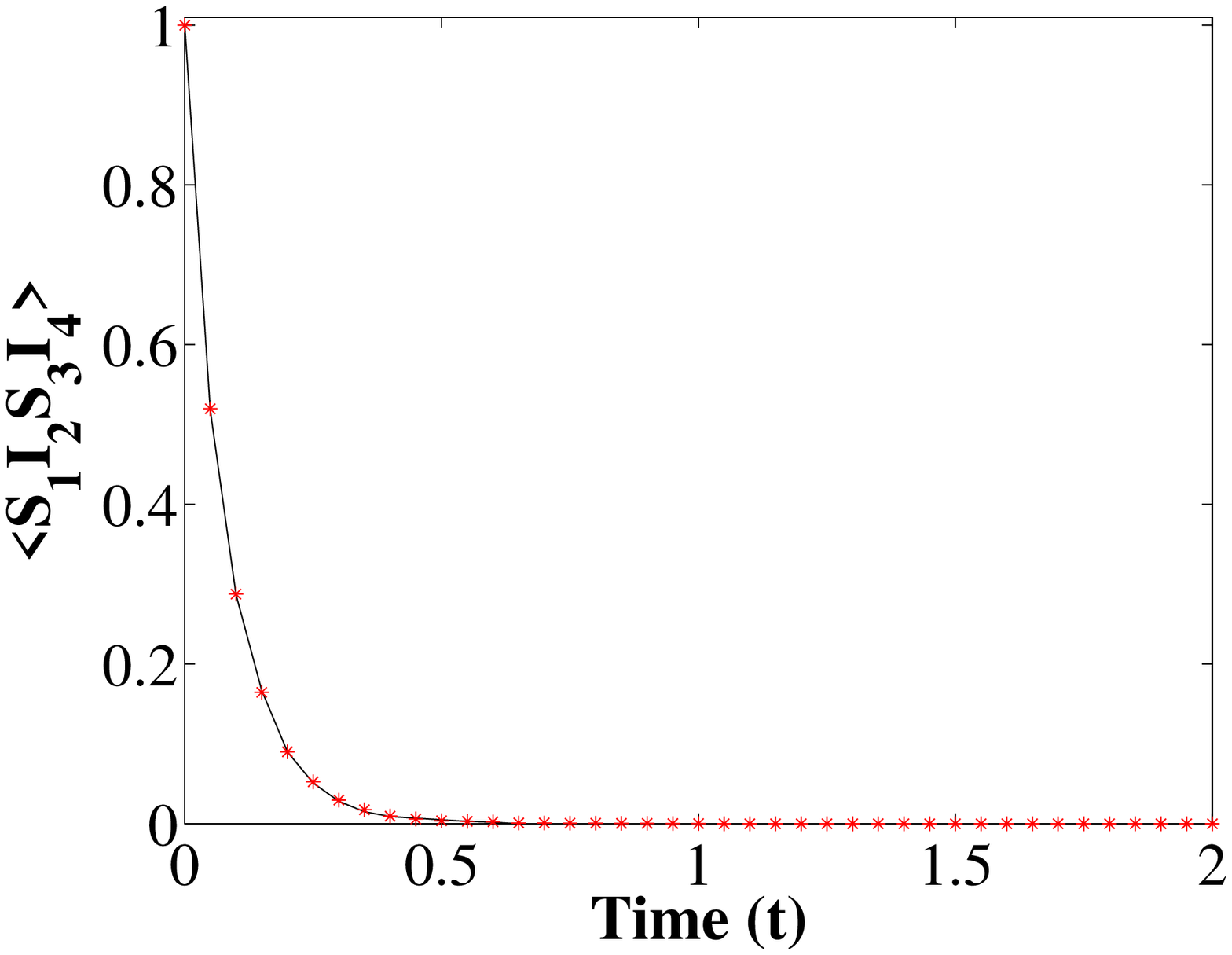,width=8cm}
\caption{\textit{Left panel:} Testing of the closure $\la S_1I_2I_3I_4\ra=\frac{\la S_1I_2I3\ra \la S_1I_4\ra}{\la S_1 \ra}$ for the toast network (see Fig.~\ref{lollipop-bowtie}d)
by plotting the left (continuous line) and right hand (dashed line) sides of the closures from the full system, with further output from Gillespie-type simulation ($\star$) by averaging  over $10^6$ simulations. \textit{Right panel:} Testing of the closure $\la S_1I_2S_3I_4\ra=\frac{\la S_1I_2S_3 \ra \la S_1S_3I_4\ra}{\la S_1S_3 \ra}$ for the same network with results exclusively based on the full system. The probability of the quadruple and the right hand side of the closure are represented by continuous line and ($\star$), respectively. System started such that  $\la S_1I_2S_3S_4\ra(0)=1$ (left panel) and $\la S_1I_2S_3I_4\ra(0)=1$ (right panel). Parameter values for all cases are $\tau=10/4$ and $\gamma=1$.}\label{test_toast_closure}
\end{figure}

%\begin{table}
%\begin{center}
%\begin{tabular}{ |c | c | c| }
%  \hline
%  & NFS & RS \\
%  \hline
%  open triple & 13 & 10  \\
%  \hline
%  4-line & 23 & 17  \\
%  \hline
%  $N$-line & $(3N^2-N+2)/2$ & $5N-3$  \\
%  \hline
%  4-star & 27 & 17  \\
%  \hline
%  lollipop & 35 & 26 \\
%  \hline
%\end{tabular}
%\end{center}
%  \caption{Reduction in the number of equations due to closures in a number of networks.}
%  \label{NoEqRedClo}
%\end{table}

\begin{table}
\begin{center}
%\begin{tabular}{ |c | c | c| }
\begin{tabular}{ | m{2.0cm} | m{3.5cm} | m{1.5cm} | }
  \hline
  & \textbf{NFS} & \textbf{RS} \\
  \hline
  	\begin{tikzpicture}
			\node[placetab] (left) at ( -0.5,0) {$$};
			\node[placetab] (middle) at ( -0.0,0) {$$};			
			\node[placetab] (right) at (0.5,0) {$$};
			\draw [-] (left) -- (middle);
			\draw [-] (right) -- (middle);			
	\end{tikzpicture} 
  & 13 & 10  \\
  \hline
    	\begin{tikzpicture}
			\node[placetab] (leftleft) at ( -0.75,0) {$$};
			\node[placetab] (left) at ( -0.25,0) {$$};			
			\node[placetab] (right) at ( 0.25,0) {$$};			
			\node[placetab] (rightright) at (0.75,0) {$$};
			\draw [-] (leftleft) -- (left);
			\draw [-] (left) -- (right);
			\draw [-] (right) -- (rightright);						
	\end{tikzpicture} 
  & 23 & 12  \\
  \hline
      	\begin{tikzpicture}
			\node[placetab] (leftleft) at ( -0.75,0) {$$};
			\node[placetab] (left) at ( -0.25,0) {$$};			
			\node[placetab] (right) at ( 0.25,0) {$$};			
			\node[placetab] (rightright) at (1,0) {$$};
			\draw [-] (leftleft) -- (left);
			\draw [-] (left) -- (right);
			\draw [dotted] (right) -- (rightright);						
	\end{tikzpicture} 
  & $(3N^2-N+2)/2$ & $5N-3$  \\
  \hline
      	\begin{tikzpicture}
  			\node[placetab] (middle) at ( 0.0,0) {$$};
			\node[placetab] (leftbottom) at ( -0.35,-0.35) {$$};			
			\node[placetab] (rightbottom) at ( 0.35,-0.35) {$$};			
			\node[placetab] (top) at (0,0.5) {$$};
			\draw [-] (middle) -- (leftbottom);
			\draw [-] (middle) -- (rightbottom);
			\draw [-] (middle) -- (top);	
	\end{tikzpicture} 
   & 27 & 17  \\
  \hline
      	\begin{tikzpicture}
  			\node[placetab] (middle) at ( 0.0,0) {$$};
			\node[placetab] (leftbottom) at ( -0.35,-0.35) {$$};			
			\node[placetab] (righttop) at ( -0.35,0.35) {$$};			
			\node[placetab] (right) at (0.5,0) {$$};
			\draw [-] (middle) -- (leftbottom);
			\draw [-] (middle) -- (righttop);
			\draw [-] (middle) -- (right);
			\draw [-] (leftbottom) -- (righttop);				
	\end{tikzpicture} 
& 35 & 26 \\
  \hline
\end{tabular}
\end{center}
  \caption{Reduction in the number of equations due to closures in a number of networks. $N$ stands for the number of nodes in a line network.}
  \label{NoEqRedClo}
\end{table}

%%%%%%%%%%%%%%%%%%%%%%
\section{Main result for networks with loops}
%%%%%%%%%%%%%%%%%%%%%%

%%%%%%%%%%%%%%%%%%%%%%
\subsection{Network structure driven closures}
%%%%%%%%%%%%%%%%%%%%%%
The analysis in this paper reveals an important relation between the structure of the network on which the epidemic is modelled and
the type of closures that are feasible. Moreover, the structural properties discussed below will also serve as a good indicator of the
feasibility of writing down exact equations for a given network. The two important structural properties are \cite{Diestel}:
\begin{defi} Let $G=\{V,E\}$ be a connected network. Let $v$ be a vertex of $G$, $v \in G(V)$. 
%Then $v$ is a \textbf{cut-vertex} iff $G \setminus \{v\}$ is disconnected.  
A node $v$ is called a \textbf{cut-vertex}, iff $G \setminus \{v\}$ is disconnected.
%A \textbf{vertex-cut} of $G$ is a set of vertices $V_C \subseteq G(V)$ such that $G \setminus V_C$ is disconnected.
\end{defi}
For our purposes we are interested in cut-vertices, i.e single nodes whose removal leads to disconnected components or subnetworks.
The second edge property that is of interest and related to the nodal property is:
\begin{defi} Let $G=\{V,E\}$ be a connected graph. An edge $e \in G(E)$ is called a \textbf{bridge} iff its removal increases the number of connected components.
It follows that an edge is a bridge iff it is not contained in any cycle, and that the end nodes of a \textbf{bridge} are \textbf{cut-vertices}.
\end{defi}
Examples of cut-vertices are provided in Fig.~\ref{lollipop-bowtie}, namely nodes $\{1\}$, $\{3\}$, $\{\{3\},\{4\}\}$ and $\{1\}$ are cut-vertices in the lollipop, bow tie, bow tie with a 
bridge and start triangle networks, respectively. Similarly, edge $(3,4)$ is a bridge in the bow tie with a bridge network. Further examples are provided in Fig.~\ref{CutVertDecomp}. We also note that for all our simple examples closures worked around cut-vertices, see the middle node in an open triple and the degree 3 node in the lollipop.

\begin{figure}
\begin{center}
\epsfig{file=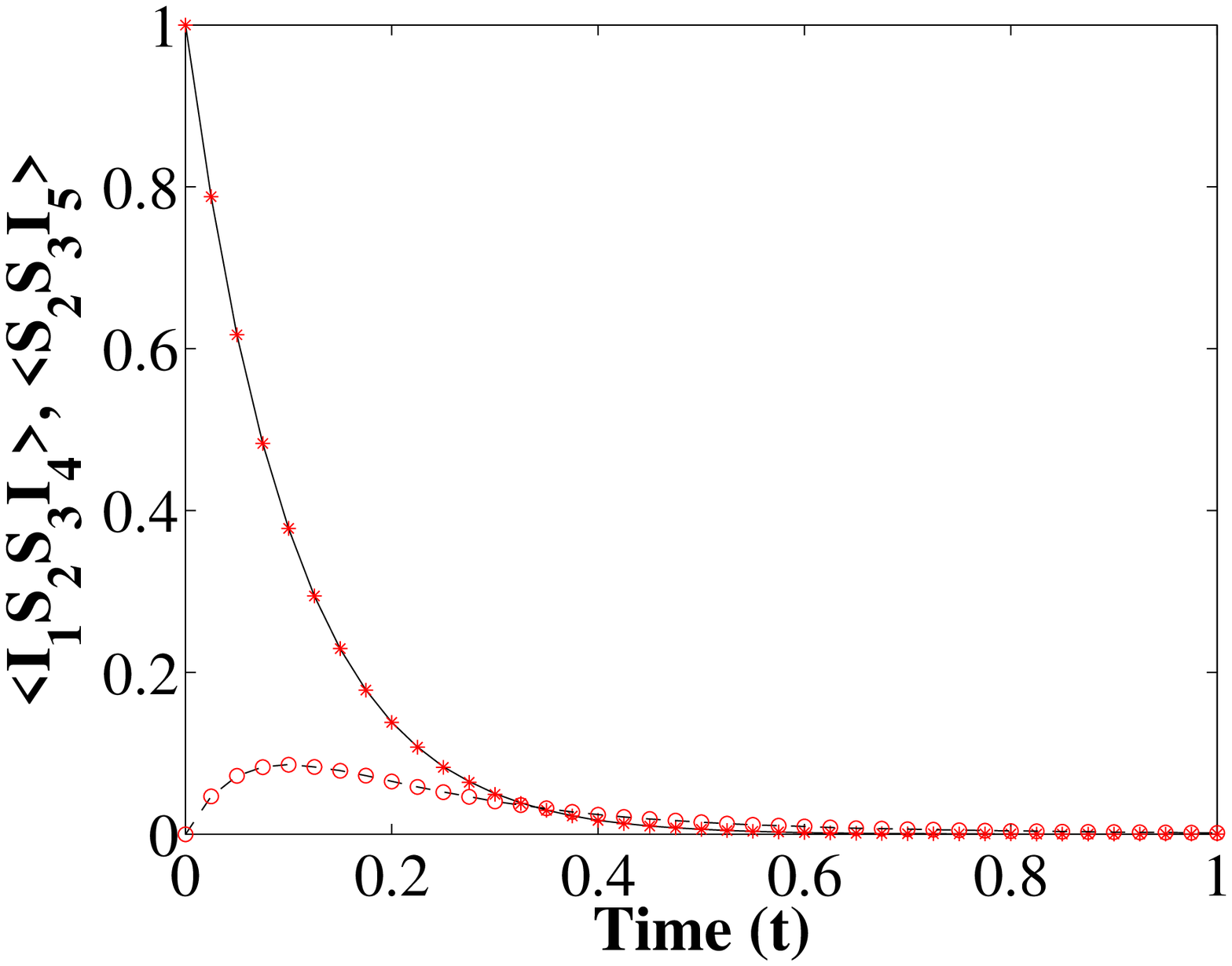,width=11cm}
\end{center}
\caption{Testing of two closures on the bow tie (see Fig.~\ref{lollipop-bowtie}(e)), ($\la I_1S_2S_3I_4\ra=\frac{\la I_1S_2S_3 \ra \la S_3I_4\ra}{\la S_3 \ra}$ and $\la S_2S_3I_5\ra=\frac{\la S_2S_3 \ra \la S_3I_5\ra}{\la S_3 \ra}$), by plotting the average of $10^5$ Gillespie-type simulations of the quadruple and the triple with continuous and dashed lines, respectively. The corresponding right hand sides of the closures are plotted with ($\star$) and ($\circ$), respectively. System started such that  $\la I_1S_2S_3I_4S_5\ra(0)=1$, and hence, $\la S_2S_3I_5\ra(0)=0$. Parameter values for all cases are $\tau=10/4$ and $\gamma=1$.}
\label{GoodCloBT}
\end{figure}

%%%%%%%%%%%%%%%%%%%%%%
\subsection{Main result}
%%%%%%%%%%%%%%%%%%%%%%
Based on the intuition gained from the closures on simple networks and their link to the structural properties of the network, via cut-vertices, we can state our main result that generalises the closures in Sharkey et. al. \cite{Sharkeyetal13Exact}, and formalises the link between closures and the structural properties of the network. This leads to the following theorem:
\begin{theo}
Let $G=\{V,E\}$ be a network with $N$ vertices ($V=\{1,2, \dots, N\}$) and a set of edges given by $E$. Consider a connected subset of vertices $F=\{v_1, v_2, \dots, v_k\} \subseteq V$, and assume that $\exists v_{i^{*}} \in F$, a cut-vertex in $G$, such that $F \setminus \{v_{i^{*}}\}$ is partitioned 
into at least two disjointed components with vertices $F_1=\{v_1, v_2, \dots, v_{i-1}\}$ and $F_2=\{v_{i+1}, v_{i+2}, \dots, v_{k}\}$ belonging to any such two, distinct and disjointed components or subnetworks. Then the following equation holds:
\begin{equation}
\Scale[1.150]{\la Z_{v_1}Z_{v_2}\cdots Z_{v_{i-1}}S_{v_{i^{*}}}Z_{v_{i+1}}Z_{v_{i+2}}\cdots Z_{v_{k}}\ra(t) = \frac{\la Z_{v_1}Z_{v_2}\cdots Z_{v_{i-1}}S_{v_{i^{*}}}\ra(t)\la S_{v_{i^{*}}}
Z_{v_{i+1}}Z_{v_{i+2}}\cdots Z_{v_{k}}\ra(t)}{\la S_{v_{i^{*}}}\ra(t)},} \label{general_closure}
\end{equation}
where $Z_{v_i} = S$ or $I$ for $\forall v_i \neq v_{i^{*}}$, and $\la \cdot \ra$ denotes the probability of a given subgraph being in a given 
state at a given time.
\end{theo}
\begin{proof} We begin by noting that:
\begin{enumerate}[i.]
\item $F$ could be the entire vertex set or a strict subset of it, and
\item By assumption, the removal of $v_{i^*}$ means that for
$\forall v_a\in\{F_1\}$ and $\forall v_b\in \{F_2\}$ there are no paths in $G$ that connect $v_a$ and $v_b$.
\end{enumerate}
By definition of conditional probabilities,
\beq
\la Z_{v_1}Z_{v_2}\cdots Z_{v_{i-1}}S_{v_{i^{*}}}Z_{v_{i+1}}Z_{v_{i+2}}\cdots Z_{v_{k}}\ra=\la Z_{v_1}Z_{v_2}\cdots Z_{v_{i-1}}S_{v_{i^{*}}}Z_{v_{i+1}}Z_{v_{i+2}}\cdots Z_{v_{k}} |
S_{v_{i^{*}}}\ra\la S_{v_{i^*}} \ra, \label{cond_prob_1}
\eeq
where the conditional probability can be written as
\begin{eqnarray}
\la Z_{v_1}Z_{v_2}\cdots Z_{v_{i-1}}S_{v_{i^{*}}}Z_{v_{i+1}}Z_{v_{i+2}}\cdots Z_{v_{N}}|S_{v_{i^{*}}}\ra&=&\la Z_{v_1}Z_{v_2}\cdots Z_{v_{i-1}}S_{v_{i^{*}}}|S_{v_{i^*}}\ra \nonumber 
\\
&&\la S_{v_{i^*}}Z_{v_{i+1}}Z_{v_{i+2}}\cdots Z_{v_{N}}|S_{v_{i^*}}\ra. \label{cond_prob_2}
\end{eqnarray}
The equality above holds due to the two subgraphs spanned by $F_1$ and $F_2$ being disjointed with no other links, except via $v_{i^*}$. Given that the cut-vertex is susceptible, it 
means that transmission via this route has not occurred, and thus the  projection of the system state on the subgraphs spanned by $F_1$ and $F_2$ must be independent.
Combining Eqs.~(\ref{cond_prob_1}-\ref{cond_prob_2}) and using that
\[ \la Z_{v_1}Z_{v_2}\cdots Z_{v_{i-1}}S_{v_{i^{*}}}|S_{v_{i^*}}\ra =\la Z_{v_1}Z_{v_2}\cdots Z_{v_{i-1}}S_{v_{i^*}}\ra / \la S_{v_{i^*}}\ra \]
and
\[ \la S_{v_{i^*}}Z_{v_{i+1}}Z_{v_{i+2}}\cdots Z_{v_{k}}|S_{v_{i^*}}\ra= \la S_{v_{i^*}}Z_{v_{i+1}}Z_{v_{i+2}}\cdots Z_{v_{k}}\ra/\la S_{v_{i^*}}\ra \] gives,
\be
\frac{\la Z_{v_1}Z_{v_2}\cdots Z_{v_{i-1}}S_{v_{i^{*}}}Z_{v_{i+1}}Z_{v_{i+2}}\cdots Z_{v_{k}}\ra}{\la S_{v_{i^*} \ra}}=\frac{\la Z_{v_1}Z_{v_2}\cdots Z_{v_{i-1}}S_{v_{i^{*}}}\ra \la 
S_{v_{i^{*}}}Z_{v_{i+1}}Z_{v_{i+2}}\cdots Z_{v_{k}}\ra}{\la S_{v_{i^*}} \ra^2},
\eeq
which is equivalent to the general closure specified in Eq.~(\ref{general_closure}).\\
\end{proof}
It is straightforward to see that all our simple intuitive closures that were exact are special cases of this main result.\\

\noindent \textbf{Special case for tree-like networks\\}
In a tree-like network all nodes (except those with degree one or one neighbour) are cut-vertices, and noting that the equations for triples in Eqs.~(\ref{SharkeyOriginal}) are all such 
that the middle node is susceptible, it follows that
\beq
\la S_iS_jI_k\ra=\frac{\la S_iS_j\ra \la S_jI_k\ra}{\la S_j \ra} \,\, \text{and} \,\, \la I_iS_jI_k\ra=\frac{\la I_iS_j\ra\la S_jI_k\ra}{\la S_j\ra}.
\ee

%%%%%%%%%%%%%%%%%%%%%%
\section{Applications of the main result}
%%%%%%%%%%%%%%%%%%%%%%
Below, we give an example for a specific network, the star-triangle network Fig.~\ref{lollipop-bowtie}g, where we write down and program in the exact system and compare the results based on the ODEs to results from Gillespie-type simulation. This is complemented by giving an upper bound on the number of equations needed for an exact representation for networks with non-overlapping loops of at most size 3, which means that overlap via edges is not allowed. More importantly, we generalise the applicability of the reduction by closure technique for arbitrary networks, and also provide an upper bound on the number of equations needed for an exact system.

%%%%%%%%%%%%%%%%%%%%%%
\subsection{Star-triangle network}
%%%%%%%%%%%%%%%%%%%%%%
Here we give details of the usefulness of the general result on the link between cut-vertices and closures. To do this
we consider the model example of a star-triangle network, as given in Fig.~\ref{lollipop-bowtie}g. 
Suppose that the triangles are labeled $1$ to $M$, the central node is labeled $1$, and the exterior nodes (i.e. the $i^{th}$ triangle except the central node) in the 
$i^{th}$ triangle are $i_1, \, i_2$. For this setup the reduced, exact system can be written down as shown in Appendix \ref{StTriEqs}. By considering the reduced system, it is straightforward to see that the number of equations depends on the number of triangles. This dependency can be quantified by a multiplicative factor which gives the number of
necessary equations for individual sub-triangles. However, we note that there is a significant difference in equation numbers when considering an isolated or single subnetwork by itself or as part of a bigger network. Namely, in this case, due to node 1 being a cut-vertex, the network will break down into $M$ disjointed triangles, upon its removal. For the purpose of generating the reduced system at the network level, these disjointed triangles  should not be considered as isolated triangles, but rather as being part of the whole network. An isolated triangle needs 18 equations (i.e. 6 equations for the nodes as each node can be $S$ or $I$, 6 equations for the edges as each edge can be $SI$ or $IS$, and 6 triples since out of the eight possible configurations with nodes being $S$ or $I$, the $SSS$ and $III$ triples are dynamically unimportant, see Appendix \ref{TriangleEqs}). However, when a triangle is considered as part of a whole network, the 18 equations need to be extended to include differential equations for $SS$-type edges, where at least one of the end nodes is a cut-vertex. This in fact accounts for infection from outside the triangle. Similarly, the equations for triples need to account for the $SSS$ triple in order to capture infection coming via a cut-vertex from outside. This extension procedure, for the current setup, requires 3 extra equations, two for the edges and one for the triples. Obviously, an equation for $\la S_{i_1}S_{i_2} \ra$ is not needed, as this pair cannot become infected from outside. Hence, to summarise, the reduced system seems to need $21 \cdot M$ equations. However, the cut-vertex (i.e. node 1), has a multiplicity $M$, and therefore, $\la S_1 \ra$ and $\la I_1 \ra$ appear $M$ times even though these are only needed once. Thus, the final number of equations in the reduced system is given by $21 \cdot M - 2 \cdot (M-1)=19 \cdot M+ 2$.

Programming in these $19 \cdot M+2$ equations in a systematic way leads to an exact representation of the $SIR$ dynamics on the star triangle network. Figure~\ref{StarTriNetw} show results from comparing the numerical solutions of the resulting system of ODEs to simulation results. The plot shows excellent agreement and supports the main result, namely that closures are exact, and hence, the ODE representation is exact. The same figure shows that distributing the same amount of initially infected nodes differently leads to different dynamics, and as expected, distributing the index cases in a more random, or less regular way, leads to a larger epidemic in this very structured network. 

\begin{figure}
\begin{center}
\epsfig{file=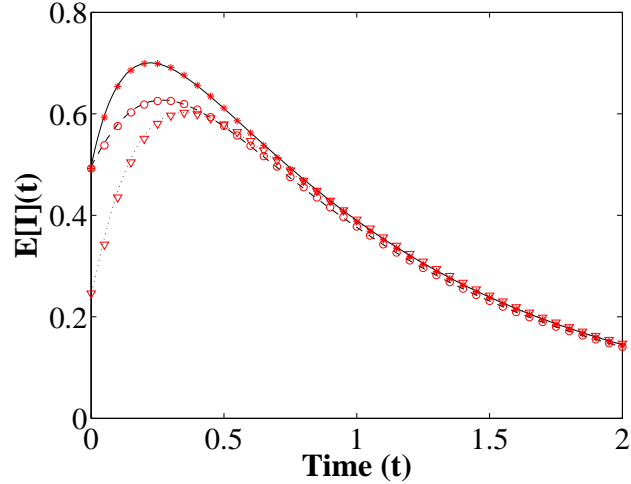,width=9cm}
\end{center}
\caption{Testing of the exact system and closures for the star network with $64$ triangles (see Fig.~\ref{lollipop-bowtie}g) by plotting the prevalence of infection based on the deterministic closed system (lines), as well as the  average of $10^4$ Gillespie-type simulations (markers). System started with the central node susceptible, and with (a) one $S$ and one $I$ node in each triangle (continuous line and ($\star$)), (b) two $S$ nodes in half of the triangles and two $I$ nodes in the rest (dashed line and ($\circ$)), and  (c) half of the triangles completely susceptible while the other half as in (a) (dotted line and ($\triangledown$)). This amounts to the same number of initially infected nodes, but distributed differently, for (a) and (b), and half as many infected nodes for (c). Parameter values for all cases are $\tau=3$ and $\gamma=1$.}
\label{StarTriNetw}
\end{figure}

%%%%%%%%%%%%%%%%%%%%%%
\subsection{General networks with loops of maximum size 3}
%%%%%%%%%%%%%%%%%%%%%%
The star-triangle discussed in the previous section satisfies this constraint. However, here we will provide a much more general statement for a wider class of networks. 
If a network has loops no larger than size 3 (which implies that triangles cannot have overlapping edges), see Figs.~\ref{lollipop-bowtie}g and \ref{TreeLikeNetwWithTri}, we can give an upper bound on the size of the system of equations describing the system dynamics.

\begin{theo}
Consider a network with $N$ nodes, $E$ edges, $T$ triangles and no larger loops. The number of equations needed to fully describe the system dynamics is less than $2N+3E+7T 
\leq 10N$.
\end{theo}

\begin{proof} We start building our system by formulating the equation for the probability of each node being either infected or susceptible. This results in $2N$ equations that depend on certain edges (or pairs) and node-state combinations, and equations for the probabilities of these specific edges are needed. One such edge-type is $\langle S_i I_j \rangle$ for all $i \neq j$, since a node $i$ can get infected by being susceptible and having an infected neighbour $j$, so both the equation for $\dot{\langle I_i \rangle}$ and $\dot{\langle S_i \rangle}$ contains the variable $\langle S_i I_j \rangle$. No other edge and node-state combination emerges from the equations for single nodes. Thus, we require equations for the probabilities of such edges, and this increase the number of equations to $2N+2E$.
Exactly as before, the equations for the edge-probabilities will involve triples. These triples can be either paths of length 3 or triangles. In the case of a path of length 3 the middle node is 
always susceptible: if we have an edge $S_i I_j$, a node $k$ can infect node $i$ (provided that it is connected to it), so we have $\langle I_k S_i I_j \rangle$ in the equation for $\dot{\langle S_i I_j \rangle}$. Similarly, the previously susceptible node $j$ can get infected by a node $l$ (provided that it is connected to it). Thus, we also have have $\langle S_i S_j I_l \rangle$ in the same equation. Since the middle node in a such a triple is a cut-vertex, our main result guarantees that the triple has an exact closure,
\[
\langle X_i S_j Y_k \rangle=\frac{\langle X_i S_j \rangle \langle S_j Y_k \rangle}{\langle S_j \rangle}.
\]
Hence, equations for such triples are not needed as they can be expressed in terms of already existing singles and edges. However, the use of these closures require extra variables and equations for these are needed. The extra variables are $\langle S_i S_j \rangle$ for all $i \neq j$ since the triples $\langle S_i S_j I_l \rangle$ for 
every $i \neq j$ are needed. This gives us $E$ extra variables and equations (we might not need all of them), and these equations contain no new variables. These in fact only depend on triples such as $S_i S_j I_l$, which we dealt with previously. Hence, this increases the number of equations to $2N+3E$. Triples can also be triangles. In this case a closure is not possible, see previous sections and Appendix (\ref{TriangleEqs}). 

\begin{figure}
 	\begin{center}
		\begin{tikzpicture}
			\node[place] (extraistvan1) at (,) {};
			
			\node[place] (lefttop) at ( -1,1) {};
			\node[place] (middle) at ( 0,0) {};
			\node[place] (5) at ( -2,2) {};
			\node[place] (6) at ( 2,2) {};
			\node[place] (7) at (3,3) {};
			\node[place] (8) at (4.3,3) {};
			\node[place] (9) at (1,3) {};
			\node[place] (10) at (1,4.3) {};
			\node[place] (11) at ( 2,0) {};
			\node[place] (12) at ( 3.3,0) {};
			\node[place] (13) at (0,5.3) {};
			\node[place] (14) at (2,5.3) {};
			%
			%\node[place] (13) at (0,5.3) {};
			\node[place] (16) at (-2,5.3) {};	
			\node[place] (17) at (-1,4.3) {};			
			\node[place] (15) at ( -3.3,2) {};
			\node[place] (righttop) at ( 1,1) {};
			\node[place] (bottom) at (0,-1.3) {};

			\draw [-] (lefttop) -- (middle);
			\draw [-] (5) -- (lefttop);
			\draw [-] (6) -- (righttop);
			\draw [-] (6) -- (7);
			\draw [-] (7) -- (8);
			\draw [-] (6) -- (9);
			\draw [-] (7) -- (9);
			\draw [-] (10) -- (9);
			\draw [-] (14) -- (13);
			\draw [-] (10) -- (13);
			\draw [-] (10) -- (14);
			\draw [-] (5) -- (15);
			\draw [-] (righttop) -- (11);
			\draw [-] (11) -- (12);
			\draw [-] (righttop) -- (middle);
			\draw [-] (lefttop) -- (righttop);
			\draw [-] (middle) -- (bottom);	
			\draw [-] (16) -- (17);			
			\draw [-] (16) -- (13);	
			\draw [-] (17) -- (13);	
		\end{tikzpicture}
	\end{center}
	\caption{An example of a network with maximum loop size 3. Overlap of triangles via edges is not allowed, but triangles can overlap via single nodes.}
	\label{TreeLikeNetwWithTri}
\end{figure}
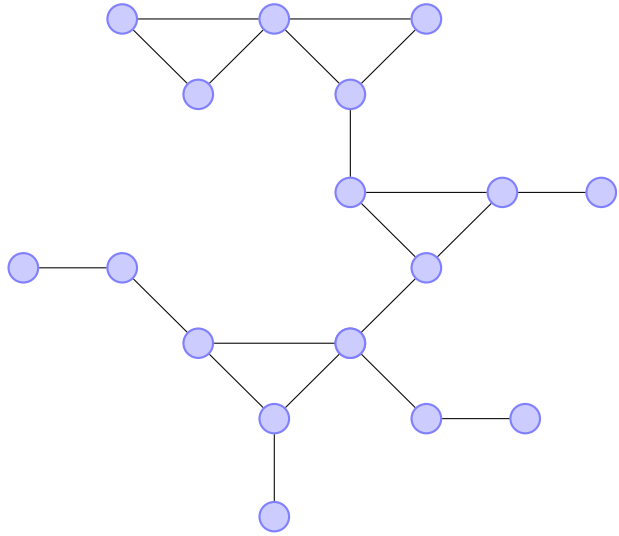			

Let us now turn our attention to triples that are triangles. Every triangle can be of 6 different types (see Appendix \ref{TriangleEqs}): take an edge of state $S_i I_j$ in the triangle. Similarly to the case of paths, the  equation for $\dot{\langle S_i I_j \rangle}$ contains $\langle I_k S_i I_j \rangle$ and $\langle S_i S_j I_k \rangle$, where $k$ is the third node of the triangle. Every type of triple containing exactly one infected node or exactly one susceptible node appears this way, because nodes $i,j,k$ are arbitrary within the triangle. In the equations for triples the probability of quadruples appear. Because of the structure of the graph these could be paths of length 4 or lollipops. But a path of length 4 cannot appear in the equation of a triple, since all of our triples that cannot be closed are triangles. A lollipop can appear in the equation of a path of length 3 or a triangle. In either case, the natural cut-vertex of the lollipop is susceptible, because if it was susceptible in the triple it stays that way, and if it was infected, it is susceptible in the lollipop and the fourth, new node is infected. So by our main result in Theorem1, these subnetworks have closures. For these closures we are going to need $\langle S_i S_j S_k \rangle $ for all $i 
\neq j \neq k$ making a triangle, since the equation of an $ S_i S_j I_k$ triple, where $k$ is  cut-vertex, will require a quadruple where infection to this triangle can come from outside.
Hence, closing such a quadruple will require a triangle with all nodes susceptible. The equations for these new triangles contain no new variables, only lollipops with an infected stem-node, which we dealt with previously. So we get $T$ more equations for these fully susceptible triples on top of the $6T$ equations for other triples and the $2N+3E$ equations thus far. Tallying all this results in $2N+3E+7T$ equations.

Much like trees, networks with loops of maximum size 3 are relatively sparse. The spanning tree of such networks has $N-1$ edges. There are some triangles in the network, and these 
triangles cannot share an edge, otherwise there would be loops of size greater than 3. So every edge is in at most one loop/trinagle. In this case, $E = N-1 + T$, since two edges of the 
triangle were already counted in the spanning tree. $T \leq \frac{N}{2}$ since to construct such a network one always has to add at least two extra nodes to get a new triangle. Summarising the above yields, 
\[
2N+3E+7T = 2N+3(N-1+T)+7T \leq 2N + 3N + \frac{3}{2}N + \frac{7}{2}N = 10N.
\]
\end{proof}

Consider for example the lollipop network. In this case, $2N+3E+7T=2\cdot 4 + 3 \cdot 4 + 7 \cdot 1=27$.  In Appendix \ref{LolEqs} we have 26 equations describing the dynamics, 
since we don't need $\langle S_3 S_4 \rangle$ for any of the closures. For the star-triangle network, see sections above and Appendix \ref{StTriEqs}, of $M$ triangles one needs a maximum of $2\cdot (1+2M)+3 \cdot 3 \cdot M + 7\cdot M=2+20 \cdot M$ equations. Actually, the exact number is smaller since we do not need $\langle S_{i_1} S_{i_2} \rangle$ in any of the triangles. So the number of equations is $2\cdot (1+2 \cdot M)+3 \cdot 3 \cdot M
+ 7 \cdot M - M =2+19 \cdot M$.

%%%%%%%%%%%%%%%%%%%%%%%%%%%
\subsection{Feasibility of the reduction by closure technique for general networks}
%%%%%%%%%%%%%%%%%%%%%%%%%%%
Here, we provide a recipe-like approach to establish the feasibility of writing down an exact representation for
a given network. To achieve this for a given network $G=\{V,E\}$, the following steps should be taken:
\begin{enumerate}
\item Find all cut-vertices of G by using the \textit{depth-first search} algorithm \cite{Sedgewick}, and denote these by $C=\{v_{i_1}, v_{i_2}, \dots, v_{i_L}\}  \subset V$. This algorithm runs in polynomial time in $(|E|+|V|)$.
\item Splice the original network into independent subnetworks (each subnetwork is well connected, but any two are disconnected) as determined by the number and properties of cut-vertices. Let us assume that this procedure leads to a family of distinct subnetworks denoted by $G_1, G_2, \dots, G_P$, where $G_i=\{V_i,E_i\}$ with $i=1, 2, \dots, P$, and each of these with frequency or counts given by $f_1, f_2, \dots, f_P$, respectively. This can be done in a way in which the cut-vertices are maintained in all
subnetworks that they generate, see Fig.~\ref{CutVertDecomp}. Let us denote by $Ind(v_{i_j})$, where $j=1, 2, \dots, L$, the number of subnetworks that cut-vertex $v_{i_j}$ belongs to. \textit{The subnetwork is in fact a generalisation of the loop concept in that it needs to be connected and with no further cut-vertices. As indicated by our results, closures within loops or subnetworks will not be exact.}
\item The relation between the distinct subnetworks $P$, their frequency, and the number of nodes in the subnetworks (e.g. $|V_1|, |V_2|, \dots, |V_P|$) will determine the number of equations needed for a full, exact representation. This relation is made more precise by the corresponding multipliers $m_1, m_2, \dots, m_P$ which simply denote the number of equations needed to describe exactly the corresponding subnetworks, e.g. an edge needs 7 equations (4 equations for the nodes and 3 equations at pair level).  In a similar way a triangle needs 22 equations (6 equations for the nodes, 9 equations for the edges and 7 equations for the triangles), a cycle graph with 4 nodes needs 45 equations, and the toast network needs 57 equations.  Hence an upper bound for the number of equations needed to describe the epidemic dynamics exactly is given by 
$$N_{EQ}(G)=\sum_{i=1}^{P}m_{i}f_{i}-2\sum_{j=1}^{L}(Ind(v_{i_j})-1).$$
\end{enumerate}
The formula simply takes a sum across the number of equations needed for all subnetworks and adjusts this to account for the unnecessary multiplications caused by cut-vertices being part of multiple subnetworks. Applying this procedure for the simplest cases of tree-like networks gives $N_{EQ}(G)=2 \cdot |V| +3 \cdot |E|$. Moreover, for tree-like networks with triangles only, the removal of all cut-vertices will leave subnetworks of two distinct types, namely the edge and triangle, yielding $N_{EQ}(G)=2 \cdot |V| +3 \cdot |E| +7\cdot T$ equations, where as before $T$ is the number of triangles in the network. The reason for $N_{EQ}(G)$ being an upper bound is due to accounting for all $SS$ pairs regardless of weather all of these are needed to link to other subnetworks. Similarly, other fully susceptible arrangements at higher level may be needed to account appropriately for outside infections (see the detailed explanation in the \textit{Star-triangle network} section). While the formula could be further improved, the exact overestimate depends in a non-trivial way on the interaction between the structure of the network and epidemic dynamics. Our investigations show that removing the unnecessary $SS$ variables will not considerably decrease the number of equations.

We note that for the decomposition of a network there are two extreme scenarios: (a) the network has many cut-vertices and few distinct subnetworks but with high frequency and (b) the network decomposes to relatively few distinct subnetwork, but of large sizes. The tree-like networks are a good example for scenario (a), where the only subnetwork is the edge, where each edge requires only 3 equations. Thus few equations per subnetwork but many subnetworks. More structured networks will typically have distinct subnetworks of larger sizes which for an exact description will require a larger number of equations. Thus fewer subnetworks, but many more equations per subnetwork. More importantly, it is non-trivial to find a simple relation between subnetworks and the number of equations needed for an exact description, and this may require further attention and work. It is straightforward to see that the desirable scenario for an exact representation is scenario (a), and it is likely that in this case an exact description is possible. Complexity quickly increases from 22 equations needed for a fully connected triangle to 57 equations for a subnetwork equivalent to a toast network, see Table \ref{IsoOrPart} that gives equation numbers for these and further examples of small networks. Thus both scenarios require a large number of equations. Generating and implementing the equations needed for an exact description is prone to error and we highly recommend the development of an algorithmic approach, where equations can be generated automatically rather than manually. While in the present project we adopted a manual, direct approach, future work will consider the implementation of an automated procedure for as general a situation as possible. The description above, illustrates clearly that the family of networks with many cut-vertices are more amenable to this approach, and it is likely that for networks with few cut-vertices, the task of writing down an exact system may be out of reach.\\

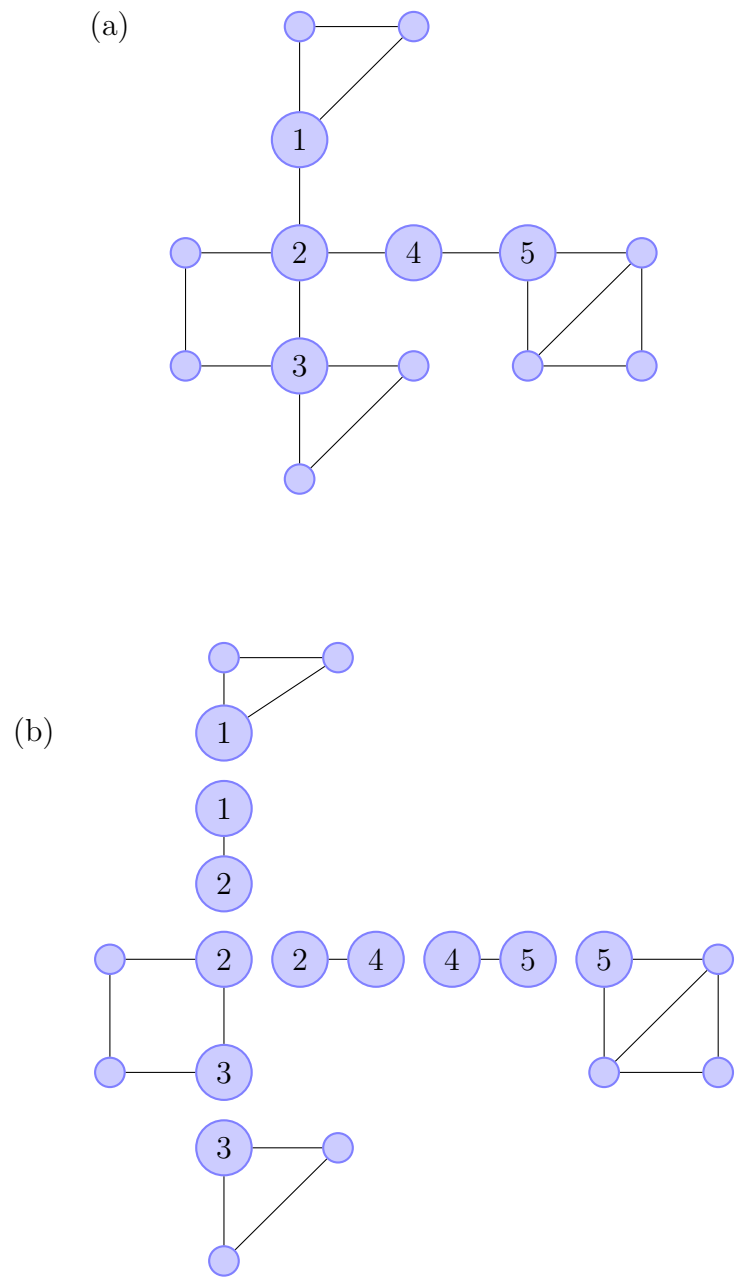
\begin{figure}
 	\begin{center}
		% showing cut-vertex decomposition
		\begin{tikzpicture}
			\node[place] (leftsqtopleft) at (-3,0) {$$};
			\node[place] (leftsqtopright) at (-1.5,0) {$2$};
			\node[place] (leftsqbotleft) at (-3,-1.5) {$$};
			\node[place] (leftsqbotright) at (-1.5,-1.5) {$3$};
			\node[place] (lowertritopright) at (0,-1.5) {$$};
			\node[place] (lowertribot) at (-1.5,-3.0) {$$};
			\node[place] (centre) at (0,0) {$4$};	
			\node[place] (uppertribot) at (-1.5,1.5) {$1 $};
			\node[place] (uppertritop) at (-1.5,3) {$ $};
			\node[place] (uppertriright) at (0,3) {$$};				
			\node[place] (rightsqtopleft) at (1.5,0) {$5$};
			\node[place] (rightsqtopright) at (3,0) {$$};
			\node[place] (rightsqbotleft) at (1.5,-1.5) {$$};
			\node[place] (rightsqbotright) at (3,-1.5) {$$};			
			\node (a) at (-4,3) {(a)};
			\draw [-] (leftsqtopleft) -- (leftsqtopright);
			\draw [-] (leftsqtopleft) -- (leftsqbotleft);
			\draw [-] (leftsqbotleft) -- (leftsqbotright);
			\draw [-] (leftsqtopright) -- (leftsqbotright);			
			\draw [-] (leftsqtopright) -- (centre);
			\draw [-] (leftsqtopright) -- (uppertribot);
			\draw [-] (leftsqbotright) -- (lowertritopright);
			\draw [-] (leftsqbotright) -- (lowertribot);											
			\draw [-] (lowertritopright) -- (lowertribot);
			\draw [-] (uppertribot) -- (uppertritop);
			\draw [-] (uppertribot) -- (uppertriright);
			\draw [-] (uppertritop) -- (uppertriright);
			\draw [-] (centre) -- (rightsqtopleft);
			\draw [-] (rightsqtopleft) -- (rightsqtopright);
			\draw [-] (rightsqtopleft) -- (rightsqbotleft);
			\draw [-] (rightsqbotleft) -- (rightsqbotright);
			\draw [-] (rightsqbotleft) -- (rightsqtopright);
			\draw [-] (rightsqbotright) -- (rightsqtopright);								
		\end{tikzpicture}
	\end{center}	
		%\hspace{1.0cm}
		\vspace{1.0cm}
	\begin{center}	
		% decomposed
		\begin{tikzpicture}
			\node[place] (leftsqtopleft) at (-3,0) {$$};
			\node[place] (leftsqtopright) at (-1.5,0) {$2$};
			\node[place] (leftsqbotleft) at (-3,-1.5) {$$};
			\node[place] (leftsqbotright) at (-1.5,-1.5) {$3$};
			\node[place] (fromsquarebotright) at (-1.5,-2.5) {$3$};			
			\node[place] (lowertritopright) at (0,-2.5) {$$};
			\node[place] (lowertribot) at (-1.5,-4.0) {$$};
			%
			%\node[place] (centre) at (0,0) {$3$};	
			%
			\node[place] (fromsquaretopright) at (-1.5, 1) {$2$};
			\node[place] (fromtribot) at (-1.5, 2) {$1$};							
			\node[place] (fromsquaretoprighttotheright) at (-0.5, 0) {$2$};
			\node[place] (fromsquaretoplefttotheleft) at (2.5, 0) {$5$};
			\node[place] (fromcentretoleft) at (0.5, 0) {$4$};
			\node[place] (fromcentretoright) at (1.5, 0) {$4$};
			\node[place] (uppertribot) at (-1.5,3) {$1$};
			\node[place] (uppertritop) at (-1.5,4) {$ $};
			\node[place] (uppertriright) at (0,4) {$$};				
			\node[place] (rightsqtopleft) at (3.5,0) {$5$};
			\node[place] (rightsqtopright) at (5,0) {$$};
			\node[place] (rightsqbotleft) at (3.5,-1.5) {$$};
			\node[place] (rightsqbotright) at (5,-1.5) {$$};			
			\node (a) at (-4,3) {(b)};
			\draw [-] (leftsqtopleft) -- (leftsqtopright);
			\draw [-] (leftsqtopleft) -- (leftsqbotleft);
			\draw [-] (leftsqbotleft) -- (leftsqbotright);
			\draw [-] (leftsqtopright) -- (leftsqbotright);			
			%\draw [-] (leftsqtopright) -- (centre);
			%\draw [-] (leftsqtopright) -- (uppertribot);
			\draw [-] (fromsquarebotright) -- (lowertritopright);
			\draw [-] (fromsquarebotright) -- (lowertribot);											
			\draw [-] (lowertritopright) -- (lowertribot);
			\draw [-] (fromsquaretopright) -- (fromtribot);
			\draw [-] (uppertribot) -- (uppertritop);
			\draw [-] (uppertribot) -- (uppertriright);
			\draw [-] (uppertritop) -- (uppertriright);
			%\draw [-] (centre) -- (rightsqtopleft);
			\draw [-] (fromsquaretoprighttotheright) -- (fromcentretoleft);
			\draw [-] (fromsquaretoplefttotheleft) -- (fromcentretoright);			
			\draw [-] (rightsqtopleft) -- (rightsqtopright);
			\draw [-] (rightsqtopleft) -- (rightsqbotleft);
			\draw [-] (rightsqbotleft) -- (rightsqbotright);
			\draw [-] (rightsqbotleft) -- (rightsqtopright);
			\draw [-] (rightsqbotright) -- (rightsqtopright);			
		\end{tikzpicture}
	\end{center}
	\caption{An example of a network with 5 cut-vertices (a), and the equivalent network upon the decomposition into subnetworks due to the removal of cut-vertices. The decomposed network has 4 subgraphs of different type: edge, triangle, cycle of size four and toast  with frequencies  3, 2, 1, 1, respectively. The five cut-vertices belong to $Ind(1)=2$, $Ind(2)=3$, $Ind(3)=2$, $Ind(4)=2$ and $Ind(5)=2$.}\label{CutVertDecomp}
\end{figure}			

%\begin{table}
%\begin{center}
%\begin{tabular}{ |c | c | c| }
%  \hline
%  & Isolated & Part of a network \\
%  \hline
%  Edge & 6 & 7  \\
%  \hline
%  Triangle & 18 & 22  \\
%  \hline
%  Square & 36 & 45  \\
%  \hline
%  Toast & 47 & 57  \\
%  \hline
%  Complete network with 4 nodes & 58 & 69  \\
%  \hline
%\end{tabular}
%\end{center}
%  \caption{The number of equations for a class of subnetworks that are isolated \textit{versus} being part of a network.}
%  \label{IsoOrPart}
%\end{table}
%

\begin{table}
\begin{center}
\begin{tabular}{ | m{1.5cm} | m{2cm} | m{4cm} |}
  %\begin{tabular}{ | c | c | c |}
  \hline
  &\textbf{Isolated} & \textbf{Part of a network} \\
  \hline  
  		\begin{tikzpicture}
			\node[placetab] (left) at ( -0.25,0) {$$};
			\node[placetab] (middle) at (0.25,0) {$$};
			\draw [-] (left) -- (middle);
		\end{tikzpicture}
   & 6 & 7  \\
  \hline
  		\begin{tikzpicture}
			\node[placetab] (leftbase) at (-0.25,0) {$$};
			\node[placetab] (rightbase) at (0.25,0) {$$};
			\node[placetab] (middletop) at (0,0.37) {$$};
			\draw [-] (leftbase) -- (rightbase);
			\draw [-] (rightbase) -- (middletop);			
			\draw [-] (leftbase) -- (middletop);					
		\end{tikzpicture}
   & 18 & 22  \\
  \hline
  			\begin{tikzpicture}
			\node[placetab] (lefttop) at ( -0.25,0.25) {$$};
			\node[placetab] (righttop) at ( 0.25,0.25) {$$};
			\node[placetab] (leftbottom) at ( -0.25,-0.25) {$$};
			\node[placetab] (rightbottom) at (0.25,-0.25) {$$};
			\draw [-] (lefttop) -- (righttop);
			\draw [-] (righttop) -- (rightbottom);
			\draw [-] (lefttop) -- (leftbottom);
			\draw [-] (leftbottom) -- (rightbottom);
			%\draw [-] (lefttop) -- (rightbottom);			
		\end{tikzpicture}
  & 36 & 45  \\
  \hline
  			\begin{tikzpicture}
			\node[placetab] (lefttop) at ( -0.25,0.25) {$$};
			\node[placetab] (righttop) at ( 0.25,0.25) {$$};
			\node[placetab] (leftbottom) at ( -0.25,-0.25) {$$};
			\node[placetab] (rightbottom) at (0.25,-0.25) {$$};
			\draw [-] (lefttop) -- (righttop);
			\draw [-] (righttop) -- (rightbottom);
			\draw [-] (lefttop) -- (leftbottom);
			\draw [-] (leftbottom) -- (rightbottom);
			\draw [-] (lefttop) -- (rightbottom);			
		\end{tikzpicture}
   & 47 & 57  \\
  \hline
  			\begin{tikzpicture}
			\node[placetab] (lefttop) at ( -0.25,0.25) {$$};
			\node[placetab] (righttop) at ( 0.25,0.25) {$$};
			\node[placetab] (leftbottom) at ( -0.25,-0.25) {$$};
			\node[placetab] (rightbottom) at (0.25,-0.25) {$$};
			\draw [-] (lefttop) -- (righttop);
			\draw [-] (righttop) -- (rightbottom);
			\draw [-] (lefttop) -- (leftbottom);
			\draw [-] (leftbottom) -- (rightbottom);
			\draw [-] (lefttop) -- (rightbottom);			
			\draw [-] (righttop) -- (leftbottom);			
		\end{tikzpicture}
   & 58 & 69  \\
  \hline
\end{tabular}
\end{center}
  \caption{The number of equations for a class of subnetworks that are isolated \textit{versus} being part of a network.}
  \label{IsoOrPart}
\end{table}

%%%%%%%%%%%%%%%%%%%%%%
\section{Discussion}
%%%%%%%%%%%%%%%%%%%%%%
In this paper we extended results for tree-like networks and Markovian $SIR$ epidemics \cite{Sharkeyetal13Exact} to networks with loops, and provided an important link between the structural properties of the network and the feasibility of writing down an exact representation of the epidemic on the network. The results are built up in a methodical way starting from the simplest networks or network motifs, and to enhance clarity we give the full system of equations whenever this is possible. The proof of the main results is in fact an alternative to and more general than the proof provided in \cite{Sharkeyetal13Exact}, and its usefulness and generality in reducing the number of equations in the exact system is illustrated by our concrete, worked out examples for tree-like networks that contain non-overlapping (via edges) loops of maximum size 3. In future work, we will concentrate on improving the upper estimate on the number of equations in a reduced, exact system, as well as making the upper estimate more explicit. 

As the size of subnetworks increases so does the number of equations and at a much faster rate. However, for cycle graphs the exact system contains relatively few equations which we now briefly explain. In the case of a cycle graph with $N$ nodes, $2N$ equations for the nodes (each node can be $S$ or $I$) and $2N$ equations for the edges (each edge can be in one of the two states $SI$ or $IS$) are needed. In the differential equations for the edges two types of triples occur, namely $SSI$ (there are $2N$ of these) and $ISI$ triples (there are $N$ of these), hence there are a total of $3N$ equations for the triples. Similarly, there are $3N$ $k$-motifs (for $k=4,5,\ldots, N-1$), namely those of type $S\ldots SI$ and those of type $IS\ldots SI$. Finally, there are a total of $2N$ $N$-motifs, $N$ of them are of type $S\ldots SI$ and $N$ of them are of type $IS\ldots SI$. Therefore, the number of equations in the full system of a cycle graph with $N$ nodes is $2N+2N+3N(N-3)+2N=3N(N-1)$. It is worth noting that this system can be lumped to $2N-1$ equations by introducing a single variable for each motif type. Namely, the first lumped variable will be $\la S \ra=\sum_{i=1}^N \la S_i \ra$ the expected number of $\la S \ra$ nodes. Similarly, the second lumped variable is $\la I \ra=\sum_{i=1}^N \la I_i \ra$ the expected number of $I$ nodes. Then in a similar way there will be a lumped variable for $k$-motifs of type $S\ldots SI$ and another lumped variable for $k$-motifs of type $IS\ldots SI$. In total the number of lumped variables will sum up to $2N-1$, meaning that the exact value of the prevalence can be given by solving a system of $2N-1$ ODEs. Hence, it is feasible to extend our results to tree-like networks with no-overlapping loops of size greater than three.

While progress in modelling epidemic dynamics on networks with loops (which usually involves clustering) has been made \cite{Ball, GleesonClusteredNetworks, Newman2009RanNetwClust, Trapman, VolzPlosOne}, many challenges remain. These challenges are both around generating clustered networks, and tuning the amount of clustering and implicitly the number and type of different loops \cite{Ball, KissComNewman, Newman2003ClustNetwGeneration, Trapman, VolzTuneDegreeClust}, and especially around providing a description of the time-evolution of the epidemic. Progress in determining the final epidemic size (time-evolution not needed) has been good but models describing the time-evolution are more challenging. In this paper, we make the first steps in providing a well-grounded and rigorous modelling alternative. Although the models presented are unlikely to pertain to analytical analysis, they could provide a valuable platform to investigate the effect of intervention or control on nodes, edges or subparts of the network. The exact system, in this case, could give precise information about the impact of isolating nodes, links or decreasing their potential of transmitting, or increasing the recovery rate of some targeted or specific nodes.

\section*{Acknowledgements}
P\'eter L. Simon acknowledges support from OTKA (grant no. 81403).
%%%%%%%%%%%%%%%%%%%%%%%%%%
\section{Appendix: The triangle network}
\label{TriangleEqs}
%%%%%%%%%%%%%%%%%%%%%%%%%%
The triangle network is a loop of three nodes numbered 1, 2 and 3, see Fig.~\ref{lollipop-bowtie}b. The system dynamics are given by the following set of equations:
\begin{align}
 \dot{\langle I_1 \rangle} &= \tau \langle S_1I_2 \rangle + \tau \langle S_1I_3 \rangle - \gamma \langle I_1\rangle , \\ \label{triangle1}
 \dot{\langle S_1 \rangle} &= -\tau \langle S_1I_2 \rangle - \tau \langle S_1I_3 \rangle,\\
 \dot{\langle I_2 \rangle} &= \tau \langle I_1S_2 \rangle + \tau \langle S_2I_3\rangle - \gamma \langle I_2\rangle ,\\
 \dot{\langle S_2 \rangle} &= -\tau \langle I_1S_2 \rangle - \tau \langle S_2I_3\rangle , \\
 \dot{\langle I_3 \rangle} &= \tau \langle I_2S_3 \rangle + \tau \langle I_1S_3 \rangle - \gamma \langle I_3 \rangle,\\
 \dot{\langle S_3 \rangle} &= -\tau \langle I_2S_3 \rangle - \tau \langle I_1S_3\rangle ,\\
 \dot{\langle I_1S_2 \rangle} &= -(\tau + \gamma) \langle I_1S_2 \rangle - \tau \langle I_1S_2I_3 \rangle + \tau \langle S_1S_2I_3 \rangle,\\
 \dot{\langle S_1I_2 \rangle} &= -(\tau + \gamma) \langle S_1I_2 \rangle + \tau \langle S_1S_2I_3 \rangle - \tau \langle S_1I_2I_3 \rangle,\\
 \dot{\langle I_2S_3 \rangle} &= -(\tau + \gamma) \langle I_2S_3 \rangle + \tau \langle I_1S_2S_3 \rangle - \tau \langle I_1I_2S_3 \rangle ,\\
 \dot{\langle S_2I_3 \rangle} &= -(\tau + \gamma) \langle S_2I_3 \rangle - \tau \langle I_1S_2I_3 \rangle + \tau \langle I_1S_2S_3 \rangle ,\\
 \dot{\langle I_1S_3 \rangle} &= -(\tau + \gamma) \langle I_1S_3 \rangle - \tau \langle I_1I_2S_3 \rangle + \tau \langle S_1I_2S_3 \rangle ,\\
 \dot{\langle S_1I_3 \rangle} &= -(\tau + \gamma) \langle S_1I_3 \rangle - \tau \langle S_1I_2I_3 \rangle + \tau \langle S_1I_2S_3 \rangle , \\
 \dot{\langle I_1S_2I_3 \rangle} &= -2(\tau + \gamma) \langle I_1S_2I_3 \rangle + \tau \langle I_1S_2S_3 \rangle + \tau \langle S_1S_2I_3 \rangle ,\\
 \dot{\langle S_1I_2I_3 \rangle} &= -2(\tau + \gamma) \langle S_1I_2I_3 \rangle + \tau \langle S_1S_2I_3 \rangle + \tau \langle S_1I_2S_3 \rangle , \\
 \dot{\langle I_1I_2S_3 \rangle} &= -2(\tau + \gamma) \langle I_1I_2S_3 \rangle + \tau \langle I_1S_2S_3 \rangle + \tau \langle S_1I_2S_3 \rangle ,\\
 \dot{\langle S_1I_2S_3 \rangle} &= -(2\tau + \gamma)\langle S_1I_2S_3 \rangle , \\
 \dot{\langle S_1S_2I_3 \rangle} &= -(2\tau + \gamma) \langle S_1S_2I_3 \rangle ,\\
 \dot{\langle I_1S_2S_3 \rangle} &= -(2\tau + \gamma) \langle I_1S_2S_3 \rangle. \label{triangle2}
\end{align}

\noindent Since this network has no cut-vertices, closures are not possible for any of the subsystems above. For instance,
\[
\la S_1S_2I_3\ra = \frac{\langle I_1S_2 \rangle \langle S_2 I_3 \rangle}{\langle S_2 \rangle}, \langle I_1S_2I_3 \rangle = \frac{\langle I_1S_2 \rangle \langle S_2 I_3 \rangle}{\langle S_2 \rangle},
\]
are closures that do not hold, see Fig.~\ref{BadCloTri}. We note that the evaluation of one of the closures above (the first) requires an extra equation for $\la S_1S_2\ra$. This is given below,
\[
\la \dot{S_1S_2\ra} = -2\tau \la S_1S_2I_3\ra.
\]

Depending on the closures that we wish to test additional equations may be needed.
%%%%%%%%%%%%%%%%%%%%%%%%%%%%%%%%
\section{Appendix: Equations for the lollipop network}
\label{LolEqs}
%%%%%%%%%%%%%%%%%%%%%%%%%%
The lollipop network we consider has nodes numbered as shown in Fig.~\ref{lollipop-bowtie}c.
%\begin{center}
%		\begin{tikzpicture}
%			\node[place] (lefttop) at ( -1,1) {$3$};
%			\node[place] (middle) at ( 0,0) {$1$};
%			\node[place] (righttop) at ( 1,1) {$4$};
%			\node[place] (bottom) at (0,-1.3) {$2$};
%			%
%			%\node (a) at (-1.5,-1.5) {(a)};
%			%
%			\draw [-] (lefttop) -- (middle);
%			\draw [-] (righttop) -- (middle);
%			\draw [-] (lefttop) -- (righttop);
%			\draw [-] (middle) -- (bottom);			
%		\end{tikzpicture}
%\end{center}
The equations describing the $SIR$ model can be formulated as follows:
\begin{align}
\dot{\langle I_1 \rangle}  &= \tau \langle S_1I_2 \rangle + \tau \langle S_1I_3 \rangle + \tau \langle S_1I_4 \rangle - \gamma \langle I_1 \rangle, \label{cannotbeclosed_first}\\
\dot{ \langle S_1 \rangle} &= -\tau \langle S_1I_2 \rangle - \tau \langle S_1I_3 \rangle - \tau \langle S_1I_4 \rangle,\\
\dot{\langle I_2 \rangle} &= \tau \langle I_1S_2 \rangle - \gamma \langle I_2 \rangle,\\
\dot{ \langle S_2 \rangle} &= -\tau \langle I_1S_2 \rangle, \\
\dot{\langle I_3 \rangle} &= \tau \langle I_1S_3 \rangle + \tau \langle S_3I_4 \rangle - \gamma \langle I_3 \rangle,\\
\dot{\langle S_3 \rangle} &= -\tau \langle I_1S_3 \rangle - \tau  \langle S_3I_4 \rangle,\\
\dot{\langle I_4 \rangle} &= \tau \langle I_1S_4 \rangle + \tau \langle I_3S_4 \rangle - \gamma \langle I_4 \rangle, \\
\dot{\langle S_4 \rangle} &= -\tau \langle I_1S_4 \rangle- \tau \langle I_3S_4 \rangle, \\
\dot{\langle I_1S_2 \rangle} &= -(\tau + \gamma) \langle I_1S_2 \rangle + \tau \langle S_1S_2I_3 \rangle + \tau \langle S_1S_2I_4 \rangle,\\
\dot{\langle S_1I_2 \rangle} &= -(\tau + \gamma) \langle S_1I_2 \rangle - \tau \langle S_1I_2I_3 \rangle - \tau  \langle S_1I_2I_4 \rangle,\\
\dot{\langle I_1S_3 \rangle} &= -(\tau + \gamma) \langle I_1S_3 \rangle + \tau \langle S_1I_2S_3 \rangle + \tau \langle S_1S_3I_4 \rangle - \tau \langle I_1S_3I_4 \rangle,\\
\dot{\langle S_1I_3 \rangle} &= -(\tau + \gamma) \langle S_1I_3 \rangle - \tau \langle S_1I_2I_3 \rangle - \tau \langle S_1I_3I_4 \rangle+\tau \langle S_1S_3I_4 \rangle,\\
\dot{\langle I_1S_4 \rangle} &= -(\tau + \gamma) \langle I_1S_4 \rangle + \tau \langle S_1I_2S_4 \rangle + \tau \langle S_1I_3S_4 \rangle - \tau \langle I_1I_3S_4 \rangle, \\
\dot{\langle S_1I_4 \rangle} &= -(\tau + \gamma) \langle S_1I_4 \rangle - \tau \langle S_1I_2I_4 \rangle - \tau \langle S_1I_3I_4 \rangle + \tau \langle S_1I_3S_4 \rangle, \\
\dot{\langle S_3I_4 \rangle} &= -(\tau + \gamma) \langle S_3I_4 \rangle + \tau \langle I_1S_3S_4 \rangle - \tau \langle I_1S_3I_4 \rangle,\\
 \dot{\langle I_3S_4 \rangle} &= -(\tau + \gamma) \langle I_3S_4 \rangle + \tau \langle I_1S_3S_4 \rangle - \tau \langle I_1I_3S_4 \rangle,\\
\dot{\langle S_1I_3I_4 \rangle} &= -2(\tau + \gamma) \langle S_1I_3I_4 \rangle + \tau \langle S_1S_3I_4 \rangle +\tau \langle S_1I_3S_4 \rangle - \tau \langle S_1I_2I_3I_4 \rangle,\\
\dot{\langle S_1S_3I_4 \rangle} &= -(2\tau + \gamma) \langle S_1S_3I_4 \rangle - \tau \langle S_1I_2S_3I_4 \rangle, \\
\dot{\langle S_1I_3S_4 \rangle} &= -(2\tau + \gamma) \langle S_1I_3S_4 \rangle - \tau \langle S_1I_2I_3S_4 \rangle, \\
\dot{\langle I_1S_3I_4 \rangle} &= -2(\tau + \gamma) \langle I_1S_3I_4 \rangle + \tau \langle S_1S_3I_4 \rangle + \tau \langle I_1S_3S_4 \rangle + \tau \langle S_1I_2S_3I_4 
\rangle, 
\\
\dot{\langle I_1I_3S_4 \rangle} &= -2(\tau + \gamma) \langle I_1I_3S_4 \rangle + \tau \langle S_1I_3S_4 \rangle + \tau \langle I_1S_3S_4 \rangle + \tau \langle S_1I_2I_3S_4 
\rangle, \\
\dot{\langle I_1S_3S_4 \rangle} &= -(2\tau + \gamma) \langle I_1S_3S_4 \rangle + \tau \langle S_1I_2S_3S_4 \rangle.\label{cannotbeclosed_last}
\end{align}
This first group of equations consist of variables (e.g. configurations of states and subgraphs) which cannot be closed or further reduced.
Naturally, this first set requires equations at the levels of triples and quadruples or full system size. Note that triples which are part of the triangle cannot be closed.
However, the second group of equations given below,

\begin{align}
%&\frac{dI_1}{dt} = \beta (S_1I_2) + \beta (S_1I_3) + \beta (S_1I_4) - \gamma I_1,\\
% &\frac{dS_1}{dt} = -\beta (S_1I_2) - \beta (S_1I_3) - \beta(S_1I_4),\\
% &\frac{dI_2}{dt} = \beta (I_1S_2)  - \beta (I_1S_2) - \gamma I_2,\\
% &\frac{dS_2}{dt} = -\beta (I_1S_2), \\
% &\frac{dI_3}{dt} = \beta (I_1S_3) + \beta (S_3I_4) - \gamma I_3,\\
% &\frac{dS_3}{dt} = -\beta (I_1S_3) - \beta (S_3I_4),\\
% &\frac{dI_4}{dt} = \beta (I_1S_4) + \beta (I_3S_4) - \gamma I_4, \\
% &\frac{dS_4}{dt} = -\beta (I_1S_4)- \beta (I_3S_4), \\
% &\frac{d(I_1S_2)}{dt} = -(\beta + \gamma) (I_1S_2) + \beta (S_1S_2I_3) + \beta(S_1S_2I_4),\\
% &\frac{d(S_1I_2)}{dt} = -(\beta + \gamma) (S_1I_2) - \beta (S_1I_2I_3) - \beta (S_1I_2I_4),\\
% &\frac{d(I_1S_3)}{dt} = -(\beta + \gamma) (I_1S_3) + \beta (S_1I_2S_3) + \beta (S_1S_3I_4) - \beta (I_1S_3I_4),\\
% &\frac{d(S_1I_3)}{dt} = -(\beta + \gamma) (S_1I_3) - \beta (S_1I_2I_3) - \beta (S_1I_3I_4)+\beta (S_1S_3I_4),\\
% &\frac{d(I_1S_4)}{dt} = -(\beta + \gamma) (I_1S_4) + \beta (S_1I_2S_4) + \beta (S_1I_3S_4) - \beta (I_1I_3S_4), \\
% &\frac{d(S_1I_4)}{dt} = -(\beta + \gamma) (S_1I_4) - \beta (S_1I_2I_4) - \beta (S_1I_3I_4) + \beta (S_1I_3S_4), \\
% &\frac{d(S_3I_4)}{dt} = -(\beta + \gamma) (S_3I_4) + \beta (I_1S_3S_4) - \beta (I_1S_3I_4),\\
%  &\frac{d(I_3S_4)}{dt} = -(\beta + \gamma)(I_3S_4) + \beta (I_1S_3S_4) - \beta (I_1I_3S_4),\\
\dot{\langle S_1I_2I_4 \rangle} &= -2(\tau + \gamma) \langle S_1I_2I_4 \rangle + \tau \langle S_1I_2I_3S_4 \rangle- \tau \langle S_1I_2I_3I_4 \rangle, \label{canbeclosed_first}\\
\dot{\langle S_1I_2I_3 \rangle} &= -2(\tau + \gamma) \langle S_1I_2I_3 \rangle + \tau \langle S_1I_2S_3I_4 \rangle- \tau \langle S_1I_2I_3I_4 \rangle,\\
%&\frac{d(S_1I_3I_4)}{dt} = -2(\beta + \gamma)(S_1I_3I_4) + \beta (S_1S_3I_4) +\beta (S_1I_3S_4) - \beta(S_1I_2I_3I_4),\\
\dot{\langle S_1S_2I_3 \rangle} &= -(\tau + \gamma) \langle S_1S_2I_3 \rangle + \tau \langle S_1S_2S_3I_4 \rangle - \tau \langle S_1S_2I_3I_4 \rangle,\\
\dot{\langle S_1S_2I_4 \rangle} &= -(\tau + \gamma) \langle S_1S_2I_4 \rangle + \tau \langle S_1S_2I_3S_4 \rangle - \tau \langle S_1S_2I_3I_4 \rangle,\\
\dot{\langle S_1I_2S_3 \rangle} &= -(\tau + \gamma) \langle S_1I_2S_3 \rangle - 2\tau \langle S_1I_2S_3I_4 \rangle, \\
\dot{\langle S_1I_2S_4 \rangle} &= -(\tau + \gamma) \langle S_1S_3I_4 \rangle - 2\tau \langle S_1I_2I_3S_4 \rangle, \\
%&\frac{d(S_1S_3I_4)}{dt} = -(2\beta + \gamma)(S_1S_3I_4) - \beta (S_1I_2S_3I_4), \\
%&\frac{d(S_1I_3S_4)}{dt} = -(2\beta + \gamma)(S_1I_3S_4) - \beta (S_1I_2I_3S_4), \\
%&\frac{d(I_1S_3I_4)}{dt} = -2(\beta + \gamma)(I_1S_3I_4) + \beta (S_1S_3I_4) + \beta (I_1S_3S_4) + \beta (S_1I_2S_3I_4), \\
%&\frac{d(I_1I_3S_4)}{dt} = -2(\beta + \gamma)(I_1I_3S_4) + \beta (S_1I_3S_4) + \beta (I_1S_3S_4) + \beta (S_1I_2I_3S_4), \\
%&\frac{d(I_1S_3S_4)}{dt} = -(2\beta + \gamma)(I_1S_3S_4) + \beta (S_1I_2S_3S_4), \\
\dot{\langle S_1I_2I_3I_4 \rangle} &= -3(\tau + \gamma) \langle S_1I_2I_3I_4 \rangle + \tau ( \langle S_1I_2I_3S_4 \rangle + \langle S_1I_2S_3I_4 \rangle ), \\
\dot{\langle S_1I_2I_3S_4 \rangle} &= -(3\tau + 2\gamma) \langle S_1I_2I_3S_4 \rangle, \\
\dot{\langle S_1I_2S_3I_4 \rangle} &= -(3\tau + 2\gamma) \langle S_1I_2S_3I_4 \rangle, \\
\dot{ \langle S_1S_2I_3I_4 \rangle} &= -2(\tau + \gamma) \langle S_1S_2I_3I_4 \rangle + \tau \langle S_1S_2S_3I_4 \rangle +\tau \langle S_1S_2I_3S_4 \rangle, \\
\dot{\langle S_1S_2I_3S_4 \rangle} &= -(2\tau + \gamma) \langle S_1S_2I_3S_4 \rangle, \\
\dot{\langle S_1S_2S_3I_4 \rangle} &= -(2\tau + \gamma) \langle S_1S_2S_3I_4 \rangle, \\
\dot{\langle S_1I_2S_3S_4) \rangle} &= -(\tau + \gamma) \langle S_1I_2S_3S_4 \rangle, \label{canbeclosed_last}
\end{align}

can be closed by using the following closures,

\begin{align}
\langle S_1I_2I_4 \rangle \langle S_1 \rangle &= \langle S_1I_2 \rangle \langle S_1 I_4 \rangle ,\label{closure_first}\\
\langle S_1I_2I_3 \rangle  \langle S_1 \rangle &=  \langle S_1I_2\rangle \langle S_1 I_3 \rangle,\\
\langle S_1S_2I_3\rangle  \langle S_1 \rangle &=  \langle S_1S_2\rangle \langle S_1 I_3\rangle ,\\
\langle S_1S_2I_4\rangle  \langle S_1 \rangle &=  \langle S_1S_2\rangle \langle S_1 I_4\rangle ,\\
\langle S_1I_2S_3\rangle  \langle S_1 \rangle &=  \langle S_1I_2\rangle \langle S_1 S_3 \rangle ,\\
\langle S_1I_2S_4 \rangle \langle S_1 \rangle &=  \langle S_1I_2\rangle \langle S_1 S_4\rangle ,\\
\langle S_1I_2I_3I_4\rangle \langle S_1 \rangle &=  \langle S_1I_2\rangle \langle S_1 I_3 I_4 \rangle ,\\
\langle S_1I_2I_3S_4 \rangle \langle S_1 \rangle &= \langle S_1I_2\rangle \langle S_1 I_3 S_4\rangle ,\\
\langle S_1I_2S_3I_4 \rangle \langle S_1 \rangle &= \langle S_1I_2\rangle  \langle S_1 S_3 I_4\rangle ,\\
\langle S_1S_2I_3I_4 \rangle \langle S_1 \rangle &= \langle S_1S_2\rangle \langle S_1 I_3 I_4 \rangle ,\\
\langle S_1S_2I_3S_4 \rangle \langle S_1 \rangle &= \langle S_1S_2\rangle \langle S_1 I_3 S_4 \rangle ,\\
\langle S_1S_2S_3I_4 \rangle \langle S_1 \rangle &= \langle S_1S_2\rangle \langle S_1 S_3 I_4 \rangle ,\\
\langle S_1I_2S_3S_4\rangle  \langle S_1 \rangle &=\langle S_1I_2\rangle \langle S_1 S_3 S_4 \rangle . \label{closure_last}
\end{align}

We note that there are two distinct types of closures. Namely, closures at the level of triples that are not part of the triangle and closures at the full system size.
To complete the closed system we need the following extra equations for variables that are required by the closures. These new variables together with their equations are,

\begin{align}
% &\dot{\langle S_1 \rangle } = -\tau \langle S_1I_2 \rangle  - \tau \langle S_1I_3 \rangle  - \tau \langle S_1I_4 \rangle ,\\
 \dot{\langle S_1 S_2 \rangle } &= -\tau \langle S_1S_2 I_3 \rangle  - \tau \langle S_1S_2 I_4 \rangle \label{extra_by_closure_first},\\
 \dot{\langle S_1 S_3 \rangle } &= -\tau \langle S_1I_2 S_3 \rangle  - 2\tau \langle S_1S_3 I_4 \rangle ,\\
 \dot{\langle S_1 S_4 \rangle } &= -\tau \langle S_1I_2S_4 \rangle  - 2\tau \langle S_1I_3 S_4 \rangle ,\\
 \dot{\langle S_1 S_3 S_4 \rangle } &= -\tau \langle S_1I_2 S_3 S_4 \rangle .\label{extra_by_closure_last}
\end{align}

By substituting the closures given in Eqs.~(\ref{closure_first}-\ref{closure_last}) into Eqs.~(\ref{canbeclosed_first}-\ref{canbeclosed_last}) together with the set of equations that 
cannot be closed, Eqs.~(\ref{cannotbeclosed_first}-\ref{cannotbeclosed_last}), and the extra variables, induced by the closures, Eqs.~(\ref{extra_by_closure_first}-
\ref{extra_by_closure_last}), will result in a system of 26 differential equations describing the system dynamics completely. Without closures, the system is fully specified by 35 
equations. Strictly speaking, we can drop the equations for $\la S_i \ra$ if we are only interested in prevalence and then the equations in the full and reduced system drop to 31 and 
23. Note that in the full system all $\la S_i \ra$s can be dropped, while in the reduced system we cannot drop $\la S_1 \ra$ as the closures rely on it.

%%%%%%%%%%%%%%%%%%%%%%%%%%
\section{Appendix: Equations for toast network}
\label{ToastEqs}
%%%%%%%%%%%%%%%%%%%%%%%%%%%
The evolution equations on the toast network labeled as in Fig.~\ref{lollipop-bowtie}d are given by
%
%\begin{center}
%\begin{tikzpicture}
%  [scale=.4,auto=left,every node/.style={circle,fill=white}]
%  \node[draw] (n1) at (-10,8) {1};
%  \node[draw] (n2) at (-5,8)  {2};
%  \node[draw] (n3) at (-10,3)  {4};
%  \node[draw] (n4) at (-5,3) {3};
%
%  \foreach \from/\to in {n1/n2,n2/n4,n3/n4,n1/n4,n1/n3}
%    \draw (\from) -- (\to);
%
%\end{tikzpicture}
%\end{center}
%
%\noindent The equations of the system dynamics:
%
\begin{align}
\dot{\langle I_1 \rangle } &= \tau \langle S_1I_2 \rangle  + \tau \langle S_1I_3 \rangle  + \tau \langle S_1I_4 \rangle  - \gamma \langle I_1 \rangle ,\\
\dot{\langle S_1 \rangle } &= -\tau \langle S_1I_2 \rangle  - \tau \langle S_1I_3 \rangle  - \tau \langle S_1I_4 \rangle ,\\
\dot{ \langle I_2 \rangle } &= \tau \langle I_1S_2 \rangle  +\tau \langle S_2I_3 \rangle   - \gamma \langle I_2 \rangle ,\\
\dot{\langle S_2 \rangle } &= -\tau \langle I_1S_2 \rangle  -\tau \langle S_2I_3 \rangle ,\\
\dot{\langle I_3 \rangle } &= \tau \langle I_1S_3 \rangle  + \tau \langle S_3I_4 \rangle  + \tau \langle I_2S_3 \rangle  - \gamma \langle I_3 \rangle ,\\
\dot{\langle S_3 \rangle } &= -\tau \langle I_1S_3 \rangle  - \tau \langle S_3I_4 \rangle  - \tau \langle I_2S_3 \rangle ,\\
\dot{\langle I_4 \rangle } &= \tau \langle I_1S_4 \rangle  + \tau \langle I_3S_4 \rangle  - \gamma \langle  I_4 \rangle , \\
\dot{\langle S_4 \rangle } &= -\tau \langle I_1S_4 \rangle - \tau \langle I_3S_4 \rangle , \\
\dot{\langle I_1S_2 \rangle } &= -(\tau + \gamma) \langle I_1S_2 \rangle  + \tau \langle S_1S_2I_3 \rangle  + \tau \langle S_1S_2I_4 \rangle  - \tau \langle I_1S_2I_3\rangle ,\\
\dot{\langle S_1I_2 \rangle } &= -(\tau + \gamma) \langle S_1I_2 \rangle  + \tau \langle S_1S_2I_3 \rangle  - \tau \langle S_1I_2I_3 \rangle  - \tau \langle S_1I_2I_4\rangle ,\\
\dot{\langle I_1S_3 \rangle } &= -(\tau + \gamma) \langle I_1S_3 \rangle  + \tau \langle S_1I_2S_3 \rangle  + \tau \langle S_1S_3I_4 \rangle  - \tau \langle I_1S_3I_4 \rangle  - \tau 
\langle I_1I_2S_3 \rangle ,\\
\dot{\langle S_1I_3 \rangle } &= -(\tau + \gamma) \langle S_1I_3 \rangle  - \tau \langle S_1I_2I_3 \rangle  - \tau \langle S_1I_3I_4 \rangle +\tau \langle S_1S_3I_4 \rangle  + \tau 
\langle S_1I_2S_3\rangle ,\\
\dot{\langle I_1S_4 \rangle } &= -(\tau + \gamma) \langle I_1S_4 \rangle + \tau \langle S_1I_2S_4 \rangle  + \tau \langle S_1I_3S_4 \rangle  - \tau \langle I_1I_3S_4 \rangle , \\
\dot{\langle S_1I_4 \rangle } &= -(\tau + \gamma) \langle S_1I_4 \rangle  - \tau \langle S_1I_2I_4 \rangle  - \tau \langle S_1I_3I_4 \rangle  + \tau \langle S_1I_3S_4 \rangle , \\
\dot{\langle S_3I_4 \rangle } &= -(\tau + \gamma) \langle S_3I_4 \rangle  + \tau \langle I_1S_3S_4 \rangle  - \tau \langle I_1S_3I_4\rangle  - \tau \langle I_2S_3I_4 \rangle ,\\
\dot{\langle I_3S_4 \rangle } &= -(\tau + \gamma) \langle I_3S_4 \rangle  + \tau \langle I_1S_3S_4\rangle  - \tau \langle I_1I_3S_4 \rangle  + \tau \langle I_2S_3S_4 \rangle ,\\
\dot{\langle S_2I_3 \rangle } &= -(\tau + \gamma) \langle S_2I_3 \rangle  + \tau \langle I_1S_2S_3 \rangle  - \tau \langle I_1S_2I_3 \rangle  + \tau \langle S_2S_3I_4 \rangle ,\\
\dot{\langle I_2S_3\rangle } &= -(\tau + \gamma) \langle I_2S_3 \rangle  + \tau \langle I_1S_2S_3 \rangle  - \tau \langle I_1I_2S_3 \rangle  - \tau \langle I_2S_3I_4\rangle ,\\
\dot{\langle S_1I_2I_4 \rangle } &= -2(\tau + \gamma) \langle S_1I_2I_4 \rangle  + \tau \langle S_1I_2I_3S_4 \rangle - \tau \langle S_1I_2I_3I_4 \rangle  + \tau \langle 
S_1S_2I_3I_4\rangle ,\\
\dot{\langle S_1I_2I_3 \rangle } &= -2(\tau + \gamma) \langle S_1I_2I_3 \rangle  + \tau \langle S_1I_2S_3I_4 \rangle - \tau \langle S_1I_2I_3I_4 \rangle  + \tau \langle S_1I_2S_3 
\rangle  + \tau \langle S_1S_2I_3 \rangle ,\\
\dot{\langle S_1I_3I_4 \rangle } &= -2(\tau + \gamma)\langle S_1I_3I_4\rangle  + \tau \langle S_1S_3I_4\rangle  +\tau \langle S_1I_3S_4 \rangle  - \tau \langle S_1I_2I_3I_4 \rangle  
+ \tau \langle S_1I_2S_3I_4\rangle ,\\
\dot{ \langle S_1S_2I_3 \rangle} &= -(2\tau + \gamma)  \langle S_1S_2I_3  \rangle + \tau  \langle S_1S_2S_3I_4  \rangle - \tau  \langle S_1S_2I_3I_4  \rangle,\\
\dot{ \langle S_1S_2I_4  \rangle} &= -(\tau + \gamma)  \langle S_1S_2I_4  \rangle + \tau \langle S_1S_2I_3S_4  \rangle - 2\tau  \langle S_1S_2I_3I_4 \rangle,\\
\dot{ \langle S_1I_2S_3  \rangle} &= -(2\tau + \gamma)  \langle S_1I_2S_3  \rangle - 2\tau  \langle S_1I_2S_3I_4  \rangle, \\
\dot{ \langle S_1S_3I_4  \rangle} &= -(2\tau + \gamma)  \langle S_1S_3I_4  \rangle - 2\tau  \langle S_1I_2S_3I_4  \rangle, \\
\dot{ \langle I_1S_3I_4  \rangle} &= -2(\tau + \gamma)  \langle I_1S_3I_4  \rangle + \tau  \langle S_1I_2S_3I_4  \rangle + \tau  \langle S_1S_3I_4  \rangle + \tau  \langle I_1S_3S_4  
\rangle - \tau  \langle I_1I_2S_3I_4  \rangle,\\
\dot{ \langle S_1I_2S_4  \rangle} &= -(\tau + \gamma)  \langle S_1I_2S_4  \rangle - 2\tau  \langle S_1I_2I_3S_4  \rangle +\tau  \langle S_1S_2I_3S_4  \rangle, \\
\dot{ \langle S_1I_3S_4  \rangle} &= -(2\tau + \gamma)  \langle S_1I_3S_4  \rangle - \tau  \langle S_1I_2I_3S_4  \rangle + \tau  \langle S_1I_2S_3S_4  \rangle, \\
\dot{ \langle I_1I_3S_4  \rangle} &= -2(\tau + \gamma)  \langle I_1I_3S_4  \rangle + \tau  \langle S_1I_3S_4  \rangle + \tau  \langle I_1S_3S_4  \rangle + \tau  \langle S_1I_2I_3S_4  
\rangle + \tau  \langle I_1I_2S_3S_4  \rangle,\\
\dot{ \langle I_1S_3S_4  \rangle} &= -(2\tau + \gamma)  \langle I_1S_3S_4  \rangle + \tau  \langle S_1I_2S_3S_4  \rangle - \tau  \langle I_1I_2S_3S_4  \rangle \\
\dot{ \langle I_1S_2I_3  \rangle} &= -2(\tau + \gamma)  \langle I_1S_2I_3  \rangle + \tau  \langle S_1S_2I_3I_4  \rangle + \tau  \langle I_1S_2S_3I_4  \rangle +\tau \la S_1S_2I_3\ra+\tau \la I_1S_2S_3 \ra, \\
\dot{ \langle I_1I_2S_3  \rangle} &= -2(\tau + \gamma)  \langle I_1I_2S_3  \rangle + \tau  \langle S_1I_2S_3I_4  \rangle + \tau  \langle I_1I_2S_3I_4  \rangle +\tau \la S_1I_2S_3 \ra+\tau \la I_1S_2S_3\ra, \\
\dot{\langle I_2S_3I_4 \rangle} &= -2(\tau + \gamma) \langle I_2S_3I_4 \rangle + \tau \langle I_1S_2S_3I_4 \rangle + \tau \langle I_1I_2S_3S_4 \rangle - \tau \langle I_1I_2S_3I_4 
\rangle, \\
\dot{\langle I_2S_3S_4 \rangle} &= -(\tau + \gamma) \langle I_2S_3S_4 \rangle + \tau \langle I_1S_2S_3S_4 \rangle  - 2\tau \langle I_1I_2S_3S_4 \rangle,\\
\dot{\langle S_2S_3I_4 \rangle} &= -(\tau + \gamma)\langle S_2S_3I_4 \rangle + \tau \langle I_1S_2S_3S_4 \rangle - 2\tau \langle I_1S_2S_3I_4 \rangle, \\
\dot{\langle I_1S_2S_3 \rangle} &= -(2\tau + \gamma) \langle I_1S_2S_3 \rangle + \tau \langle S_1S_2S_3I_4 \rangle - \tau \langle I_1S_2S_3I_4 \rangle, \\
\dot{\langle S_1I_2I_3I_4 \rangle} &= -3(\tau + \gamma) \langle S_1I_2I_3I_4 \rangle + \tau ( \langle S_1I_2I_3S_4 \rangle + 2\langle S_1I_2S_3I_4 \rangle + \langle S_1S_2I_3I_4 
\rangle ), \\
\dot{\langle S_1I_2I_3S_4 \rangle} &= -(3\tau + 2\gamma) \langle S_1I_2I_3S_4 \rangle + \tau (\langle S_1S_2I_3S_4 \rangle +\langle S_1I_2S_3S_4 \rangle ), \\
\dot{\langle S_1I_2S_3I_4 \rangle} &= -(4\tau + 2\gamma) \langle S_1I_2S_3I_4 \rangle, \\
\dot{\langle S_1S_2I_3I_4 \rangle} &= -(3\tau + 2\gamma) \langle S_1S_2I_3I_4 \rangle + \tau \langle S_1S_2S_3I_4 \rangle +\tau \langle S_1S_2I_3S_4 \rangle, \\
\dot{\langle S_1S_2I_3S_4 \rangle} &= -(3\tau + \gamma) \langle S_1S_2I_3S_4 \rangle,\\
\dot{\langle S_1S_2S_3I_4 \rangle} &= -(2\tau + \gamma) \langle S_1S_2S_3I_4 \rangle, \\
\dot{\langle S_1I_2S_3S_4 \rangle} &= -(2\tau + \gamma) \langle S_1I_2S_3S_4 \rangle, \\
\dot{\langle I_1S_2S_3S_4 \rangle} &= -(3\tau + \gamma) \langle I_1S_2S_3S_4 \rangle,\\
\dot{\langle I_1S_2S_3I_4 \rangle} &= -(3\tau + 2\gamma) \langle I_1S_2S_3I_4 \rangle +\tau \langle I_1S_2S_3S_4 \rangle + \tau \langle S_1S_2S_3I_4 \rangle, \\
\dot{\langle I_1I_2S_3S_4 \rangle} &= -(3\tau + 2\gamma) \langle I_1I_2S_3S_4 \rangle +\tau \langle I_1S_2S_3S_4 \rangle + \tau \langle S_1I_2S_3S_4 \rangle, \\
\dot{\langle I_1I_2S_3I_4 \rangle} &= -3(\tau + \gamma) \langle I_1I_2S_3I_4 \rangle +\tau \langle I_1S_2S_3I_4 \rangle + \tau \langle I_1I_2S_3S_4 \rangle + 2\tau \langle 
S_1I_2S_3I_4 \rangle.
\end{align}

%%%%%%%%%%%%%%%%%%%%%%%%%%%%%%
\section{Appendix: Equations for the star-triangle network}
\label{StTriEqs}
%%%%%%%%%%%%%%%%%%%%%%%%%%%%%%
In this section we write down the system of differential equations that are an exact representation of the $SIR$ epidemic on the star triangle-network, see Fig.~\ref{lollipop-bowtie}g. The relevant equations are:
\begin{align}
\dot{\langle S_1 \rangle} &= - \tau \sum_{j=1}^M\sum_{k=1}^2 \langle S_1I_{j_k} \rangle,\\
\dot{ \langle I_{1} \rangle} &= +  \tau \sum_{j=1}^M\sum_{k=1}^2 \langle S_1I_{j_k} \rangle -\gamma \la I_1 \ra,\\
\dot{ \langle S_{i_1} \rangle} &= -  \tau  \langle I_1S_{i_1} \rangle - \tau  \langle S_{i_1}I_{i_2} \rangle,\label{st1}\\
\dot{ \langle I_{i_1}  \rangle} &=    \tau  \langle I_1S_{i_1} \rangle + \tau  \langle S_{i_1}I_{i_2} \rangle - \gamma  \langle I_{i_1} \rangle,\\
\dot{ \langle S_{i_2} \rangle} &= -   \tau  \langle I_1S_{i_2} \rangle - \tau  \langle I_{i_1}S_{i_2} \rangle,\\
\dot{ \langle I_{i_2} \rangle} &=    \tau  \langle I_1S_{i_2}  \rangle + \tau  \langle I_{i_1}S_{i_2} \rangle - \gamma   \langle I_{i_2}  \rangle,\\
\dot{ \langle S_{i_1}I_{i_2} \rangle} &= -(\tau+\gamma)  \langle S_{i_1}I_{i_2} \rangle +  \tau  \langle I_1S_{i_1}S_{i_2}  \rangle - \tau  \langle I_1S_{i_1}I_{i_2} \rangle, \\
\dot{ \langle I_{i_1}S_{i_2} \rangle} &= -(\tau+\gamma) \langle I_{i_1}S_{i_2} \rangle +  \tau  \langle I_1S_{i_1}S_{i_2} \rangle - \tau  \langle I_1I_{i_1}S_{i_2} \rangle, \\
\dot{ \langle I_1S_{i_1} \rangle} &= -(\tau+\gamma) \langle I_1S_{i_1} \rangle +  \tau  \langle S_1S_{i_1}I_{i_2} \rangle - \tau  \langle I_1S_{i_1}I_{i_2} \rangle + \tau \sum_{j=1, j \neq i}^M\sum_{k=1}^2\frac{ \langle S_1S_{i_1} \rangle \langle S_1I_{j_k} \rangle}{ \langle S_1 \rangle}, \label{ex1} \\
\dot{ \langle I_1S_{i_2} \rangle} &= -(\tau+\gamma) \langle I_1S_{i_2} \rangle +  \tau  \langle S_1I_{i_1}S_{i_2} \rangle - \tau  \langle I_1I_{i_1}S_{i_2} \rangle +  \tau \sum_{j=1, j \neq i}^M\sum_{k=1}^2\frac{ \langle S_1S_{i_2} \rangle \langle S_1I_{j_k} \rangle}{ \langle S_1 \rangle},  \\
\dot{ \langle S_1 I_{i_1} \rangle} &= - (\tau+\gamma)  \langle S_1I_{i_1} \rangle -\tau \langle S_1 I_{i_1} I_{i_2} \rangle + \tau  \langle S_1 S_{i_1} I_{i_2} \rangle - \tau \sum_{j=1, j 
\neq i}^M\sum_{k=1}^2\frac{ \langle S_1 I_{i_1} \rangle  \langle S_1 I_{j_k} \rangle}{ \langle S_1 \rangle},\\
\dot{ \langle S_1 I_{i_2} \rangle} &= - (\tau+\gamma)  \langle S_1I_{i_2} \rangle - \tau \langle S_1 I_{i_1} I_{i_2} \rangle + \tau  \langle S_1 I_{i_1} S_{i_2} \rangle -  \tau \sum_{j=1, j 
\neq i}^M\sum_{k=1}^2\frac{ \langle S_1I_{i_2} \rangle  \langle S_1 I_{j_k} \rangle}{ \langle S_1 \rangle},\\
\dot{ \langle S_1 S_{i_1} \rangle} &= - \tau  \langle S_1S_{i_1}I_{i_2} \rangle-  \tau \sum_{j=1, j \neq i}^M\sum_{k=1}^2\frac{ \langle S_1S_{i_1} \rangle  \langle S_1 I_{j_k} \rangle}
{ \langle S_1 \rangle},\\
\dot{ \langle S_1 S_{i_2} \rangle} &= - \tau  \langle S_1I_{i_1}S_{i_2}  \rangle-  \tau \sum_{j=1, j \neq i}^M\sum_{k=1}^2\frac{ \langle S_1S_{i_2} \rangle  \langle S_1 I_{j_k} \rangle}
{ \langle S_1 \rangle},\\
\dot{ \langle I_1S_{i_1}S_{i_2} \rangle} &= -(2\tau+\gamma) \langle I_1S_{i_1}S_{i_2} \rangle +  \tau \sum_{j=1, j \neq i}^M\sum_{k=1}^2\frac{ \langle S_1S_{i_1}S_{i_2} \rangle \langle 
S_1I_{j_k} \rangle}{ \langle S_1 \rangle}, \\
\dot{ \langle I_1S_{i_1}I_{i_2} \rangle} &= -2(\tau+\gamma) \langle I_1S_{i_1}I_{i_2} \rangle +  \tau\sum_{j=1, j \neq i}^M\sum_{k=1}^2\frac{ \langle S_1S_{i_1}I_{i_2} \rangle \langle 
S_1I_{j_k} \rangle}{ \langle S_1 \rangle} \nonumber \\
& \quad + \tau  \langle I_1S_{i_1}S_{i_2} \rangle + \tau  \langle S_1S_{i_1}I_{i_2} \rangle, \\
\dot{ \langle I_1I_{i_1}S_{i_2} \rangle} &= -2(\tau+\gamma) \langle I_1I_{i_1}S_{i_2} \rangle + \tau \sum_{j=1, j \neq i}^M\sum_{k=1}^2\frac{ \langle S_1I_{i_1}S_{i_2} \rangle \langle 
S_1I_{j_k} \rangle}{ \langle S_1 \rangle} \nonumber \\
& \quad + \tau  \langle I_1S_{i_1}S_{i_2}  \rangle + \tau  \langle S_1I_{i_1}S_{i_2} \rangle, \\
%\frac{dS_{1}}{dt} &= -  \beta \sum_{j=1}^M\sum_{k=1}^2(S_1I_{j_k}),\\
\dot{ \langle S_1 S_{i_1} I_{i_2} \rangle} &= - (2 \tau + \gamma) \langle S_1 S_{i_1} I_{i_2} \rangle - \tau \sum_{j=1, j \neq i}^M\sum_{k=1}^2\frac{ \langle S_1S_{i_1} I_{i_2} \rangle  
\langle S_1 I_{j_k} \rangle}{ \langle S_1  \rangle},\\
\dot{ \langle S_1 I_{i_1} S_{i_2} \rangle} &= - (2 \tau + \gamma) \langle S_1 I_{i_1} S_{i_2} \rangle -  \tau \sum_{j=1, j \neq i}^M\sum_{k=1}^2\frac{ \langle S_1I_{i_1} S_{i_2} \rangle  
\langle S_1 I_{j_k} \rangle}{ \langle S_1  \rangle},\\
\dot{ \langle S_1I_{i_1}I_{i_2} \rangle} &= -2(\tau+\gamma) \langle S_1I_{i_1}I_{i_2} \rangle +  \tau \la S_1S_{i_1}I_{i_2}\ra +\tau \la S_1I_{i_1}S_{i_2}\ra \nonumber \\
& - \tau\sum_{j=1, j \neq i}^M\sum_{k=1}^2\frac{ \langle S_1I_{i_1}I_{i_2} \rangle \langle 
S_1I_{j_k} \rangle}{ \langle S_1 \rangle}, \\
\dot{ \langle S_1 S_{i_1} S_{i_2}  \rangle} &= - \tau \sum_{j=1, j \neq i}^M\sum_{k=1}^2\frac{ \langle S_1S_{i_1} S_{i_2} \rangle  \langle S_1 I_{j_k} \rangle}{ \langle S_1  \rangle},\label{ex2}
\end{align}
where, in Eqs.~(\ref{ex1} -- \ref{ex2}) we have used the following closures:
$$ \langle S_1S_{i_1}S_{i_2}I_{k_j}  \rangle \langle S_1  \rangle=   \langle S_1S_{i_1}S_{i_2} \rangle  \langle S_1I_{k_j} \rangle,$$
$$ \langle S_1S_{i_1}I_{k_j} \rangle \langle S_1  \rangle= \langle S_1S_{i_1} \rangle \langle S_1I_{k_j}\rangle,$$
$$ \langle S_1S_{i_2}I_{k_j}\rangle \langle S_1  \rangle= \langle S_1S_{i_2}\rangle \langle S_1I_{k_j}\rangle,$$
$$\langle S_1S_{i_1}I_{i_2}I_{k_j} \rangle \langle S_1  \rangle=  \langle S_1S_{i_1}I_{i_2}\rangle \langle S_1I_{k_j}\rangle,$$ and
$$\langle S_1I_{i_1}S_{i_2}I_{k_j} \rangle \langle S_1  \rangle=  \langle S_1I_{i_1}S_{i_2}\rangle \langle S_1I_{k_j}\rangle,$$ for $i, k=1, 2, \dots, M$, $i \neq k$, and $j=1, 2$.\\
These closures are of two main type, namely:
\begin{enumerate}
\item closure of a triple which is not a triangle ($i\neq j$):
\begin{equation}
\langle X_{i_l}S_1Y_{j_k} \rangle \langle  S_1 \rangle =\langle X_{i_l}S_1 \rangle \langle S_1Y_{j_k}\rangle
\end{equation}
where $l,k=1,2$, $i,j=1,\dots ,M$ and $X, Y$ are either $S$ or $I$ in some particular combination.
\item closure of a quadruple containing a triangle:
\begin{equation}
\langle X_{i_1}Y_{i_2}S_1Z_{j_k} \rangle \langle S_1 \rangle  =\langle X_{i_1}Y_{i_2}S_1\rangle \langle S_1Z_{j_k}\rangle
\end{equation}
where $k=1,2$, $i,j=1,\dots ,M$ and $X, Y, Z$ are either $S$ or $I$ in some particular combination.
\end{enumerate}

\end{document}